\documentclass[final]{siamart220329}

\usepackage{mathtools}
\usepackage{amsfonts}
\usepackage{graphicx,booktabs,array}
\usepackage{epstopdf}
\usepackage{float}
\usepackage{arydshln}
\usepackage{makecell}
\usepackage{tikz}
\usepackage{subfigure}
\usepackage{algorithmic}
\Crefname{ALC@unique}{Line}{Lines} % <- Preamble
\ifpdf
\DeclareGraphicsExtensions{.eps,.pdf,.png,.jpg}
\else
\DeclareGraphicsExtensions{.eps}
\fi

\usepackage{textcase}
\let\Up\MakeUppercase
\usepackage{cite} % to sort the citations
\usepackage{dsfont}
\usepackage{enumerate}
\usepackage[shortlabels]{enumitem}
\usepackage{xcolor}
\definecolor{darkred}{RGB}{175, 0, 0}
\definecolor{niceblue}{RGB}{0, 102, 204} 
\definecolor{nicegreen}{RGB}{76, 175, 80} 
\definecolor{LightGrey}{RGB}{211, 211, 211} 
\definecolor{darkred}{rgb}{0.8, 0.0, 0.0}

\usepackage{todonotes}
\usepackage{tikz}
\usetikzlibrary{shapes, positioning, arrows, bending, calc}
\newsiamremark{remark}{Remark}
\newsiamremark{hypothesis}{Hypothesis}
\crefname{hypothesis}{Hypothesis}{Hypotheses}
\newsiamthm{claim}{Claim}
\newsiamthm{example}{Example}
\newcommand{\N}{\mathbb{N}}
\newcommand{\R}{\mathbb{R}}

\newcommand{\argmin}{\operatorname{argmin}}

\newcommand{\bu}{{\boldsymbol{u}}}
\newcommand{\bv}{{\boldsymbol{v}}}
\newcommand{\bx}{{\boldsymbol{x}}}
\newcommand{\by}{{\boldsymbol{y}}}

\newcommand{\graphLdelta}{\Delta_{\Psi^\delta_\Theta}}
\newcommand{\graphLdeltak}{\Delta_{\Psi^{\delta_k}_{\Theta_k}}}
\newcommand{\graphL}{\Delta_{\bx_0}}
\newcommand{\sol}{\bx^{\delta}_{\Psi^\delta_\Theta,\alpha}}

\newcommand{\gt}{\bx_{\textnormal{gt}}}

\newcommand{\graphLa}{$\texttt{graphLa+}\Psi$ }

\definecolor{darkred}{rgb}{0.8, 0.0, 0.0}

\headers{\graphLa}{Bianchi et al.}

\title{A data-dependent regularization method based on the graph Laplacian}

\author{Davide Bianchi{\thanks{\MakeLowercase{\Up{S}chool of \Up{M}athematics (\Up{Z}huhai), \Up{S}un \Up{Y}at-sen \Up{U}niversity, \Up{Z}huhai, 519082, \Up{C}hina (\email{bianchid@mail.sysu.edu.cn}). }}}
			\and Davide Evangelista{\thanks{\MakeLowercase{\Up Department of \Up Computer \Up Science and \Up Engineering, \Up University of \Up Bologna, \Up Bologna, 40126, \Up Italy\\ (\email{davide.evangelista5@unibo.it}, \email{elena.loli@unibo.it}).}}}
			\and Stefano Aleotti{\thanks{\MakeLowercase{\Up Department of \Up Science and \Up High \Up Technology, \Up University of \Up Insubria, \Up Como, 22100, \Up Italy 
					\\ (\email{saleotti@uninsubria.it}, \email{marco.donatelli@uninsubria.it}).}}}
	\and Marco Donatelli\footnotemark[3]
\and Elena Loli Piccolomini\footnotemark[2]
\and Wenbin Li{\thanks{\MakeLowercase{\Up{S}chool of \Up{S}cience, \Up{H}arbin \Up{I}nstitute of \Up{T}echnology, \Up{S}henzhen,  \Up{S}henzhen, 518055, \Up{C}hina \\(\email{liwenbin@hit.edu.cn})}. \funding{\MakeLowercase{\Up{D}avide \Up{B}ianchi is supported by \Up{NSFC} (\Up{G}rant \Up{N}o.~12250410253) and by the \Up{S}tartup \Up{F}und of \Up{S}un \Up{Y}at-sen \Up{U}niversity.  \Up{M}arco \Up{D}onatelli is partially supported by \Up{MIUR - PRIN 2022 N.2022ANC8HL} and \Up{GNCS-INdAM}. \Up{W}enbin \Up{L}i  is supported by the \Up{F}undamental \Up{R}esearch \Up{F}unds for the \Up{C}entral \Up{U}niversities (grant no. \Up{HIT.OCEF.2024017}), and the \Up{U}niversity \Up{I}nnovative \Up{T}eam \Up{P}roject of \Up{G}uangdong (grant no. \Up{2022KCXTD039}).}}}}}

\usepackage{amsopn}

\ifpdf
\hypersetup{
	pdftitle={The \graphLa method},
	pdfauthor={D. Bianchi, D. Evangelista, M. Donatelli,  E. Loli Piccolomini, and W. Li}
}
\fi

\begin{document}
	\maketitle
	\begin{abstract}
			We investigate a variational method for ill-posed problems, named $\texttt{graphLa+}\Psi$, which embeds a graph Laplacian operator in the regularization term. The novelty of this method lies in constructing the graph Laplacian based on a preliminary approximation of the solution, which is obtained using any existing reconstruction method $\Psi$ from the literature. As a result, the regularization term is both dependent on and adaptive to the observed data and noise.   We demonstrate that  $\texttt{graphLa+}\Psi$ is a regularization method and  rigorously establish both its convergence and stability properties.
	
	We present selected numerical experiments in 2D computerized tomography, wherein we integrate the $\texttt{graphLa+}\Psi$ method with various reconstruction techniques $\Psi$, including Filter Back Projection (\texttt{graphLa+FBP}), standard Tikhonov (\texttt{graphLa+Tik}), Total Variation (\texttt{graphLa+TV}), and a trained deep neural network (\texttt{graphLa+Net}). The $\texttt{graphLa+}\Psi$ approach significantly enhances the quality of the approximated solutions for each method $\Psi$. Notably, \texttt{graphLa+Net} is outperforming, offering a robust and stable application of deep neural networks in solving inverse problems.

	\end{abstract}

% REQUIRED

% REQUIRED

\section{Introduction}\label{sec:introduction}
We consider the model equation
\begin{equation}\label{eq:Model_Equation}\tag{ME}
	K\bx = \by^\delta, 
\end{equation}
where $K \colon X\simeq \R^n \to Y\simeq \R^m$ represents a discretized version of a linear operator that is inherently ill-posed. Given a fixed $\gt \in X$ and $\by\coloneqq K\gt$, we want to recover a good approximation of the ground-truth $\gt$ from a noisy observation $\by^\delta$ of $\by$, that is
\begin{equation}\label{eq:problem_formulation}
	\by^\delta \coloneqq  \by + \boldsymbol{\eta}, \qquad \|\by^\delta -  \by\|\leq \delta,
\end{equation}	
where $\boldsymbol{\eta}$ is a random perturbation,  typically unavoidable in practical scenarios, and $\delta>0$ is the noise intensity.

The ill-posedness of $K$ and the presence of noise force the introduction of a regularization strategy to solve our model equation. A  standard variational method for \cref{eq:Model_Equation} reads
\begin{equation}\label{model_eq2}
	\bx^\delta_\alpha \in	\underset{\bx \in X}{\argmin} \left\{ \frac{1}{2}\|K\bx - \by^\delta \|_2^2 + \alpha\|D\bx\|_1 \right\},
\end{equation}
where $D$ is a linear mapping, characterized by the property that $\ker(D) \cap \ker(K) =~\{\mathbf{0}\}$, see \cite[Chapter 8]{engl1996regularization} and \cite{morozov2012methods}. In this context, the term $\frac{1}{2}\|K\mathbf{x} - \mathbf{y}^\delta \|_2^2$ quantifies the fidelity of the reconstruction, the $\|D\mathbf{x}\|_1$ component serves as regularization term, and $\alpha >0$ balances the trade-off between data fidelity and the regularization effect. Formulation \eqref{model_eq2} is a special case of generalized Tikhonov.

Typical choices for $D$ include linear differential operators, especially for denoising applications involving signals that are nearly piecewise-constant, such as those encountered in imaging \cite{hansen2006deblurring}.

In the classical  approach, $D$ is a discretization of Euclidean first or second-order linear differential operators, see e.g. \cite{ROF92,DR14}. 
However, the last decade has seen an increasing interest in nonlocal models and techniques from graph theory \cite{gilboa2007nonlocal,gilboa2009nonlocal,Peyre2008nonlocal,arias2009variational}. This approach has been further investigated in the context of image deblurring and computerized tomography problems in \cite{bianchi2021graph_approximation,buccini2021graph,bianchi2021graph,lou2010image,zhang2010bregmanized}.

In essence, embedding a graph-based operator $D$ in the regularization term acts as a guiding mechanism for the overall regularization process. This operator helps in identifying the ``correct" neighborhood to concentrate the reconstruction efforts on, by capturing specific features that can be inferred from the observed signal $\by^\delta$.

Initially, a graph structure $G$ is constructed from a discretized signal to incorporate features such as interfaces and discontinuities. This discrete space, which heavily depends on the signal itself, can provide more insights into the neighborhood where $\gt$ resides than a flat manifold like the Euclidean space.

Then, by choosing a suitable graph operator $D$, the optimization process in Equation \eqref{model_eq2} is oriented towards the (supposed) neighborhood of $\gt$.

 The key point is to construct the graph from a signal that closely approximates the primary features of $\gt$. In \cite{lou2010image}, it was observed that generating the graph $G$ directly from the observed and noisy data $\by^\delta$ results in poor outcomes for imaging tasks such as deblurring or tomographic reconstruction. This is because $\by^\delta$ exists in a different domain compared to $\gt$. To address this, the authors suggested an initial preprocessing step, transforming $\by^\delta$ to $\Psi(\by^\delta)$, and subsequently constructing a graph from $\Psi(\by^\delta)$. This preprocessing step involves a reconstruction map $\Psi \colon Y \to X$,  from the space of observations $Y$ to the domain $X$ where $\gt$ lives. This can be achieved, for example, by employing a standard Tikhonov filter \cite{engl1996regularization} or the Filter Back Projection (FBP) method \cite{kak2001principles}, depending on the inverse problem to handle. In the same spirit and independently,  iterative schemes incorporating updates to graph weights were introduced in \cite{Peyre2008nonlocal,zhang2010bregmanized,arias2009variational}.
 
 Taking inspiration from \cite{lou2010image}, we introduce the novel method \graphLa. The initial preprocessing step involves selecting a family of reconstructor maps 
 $$\Psi_\Theta \colon Y \to X,$$ 
 where $\Theta= \Theta(\delta, \by^\delta)$ is a family of parameters that can depend on $\delta$ and $\by^\delta$. It is important to note that the pair $(\Psi_\Theta, \Theta)$ is very general and may not be a convergent regularizing method.

 Then, a graph $G$ is built from  $\Psi^\delta_\Theta\coloneqq\Psi_\Theta (\by^\delta)\in X$, as well as its associated graph Laplacian operator $\graphLdelta$ (see Section \ref{ssec:graph_theory} for a proper definition). Finally, we find a minimizer of \eqref{model_eq2}  with $D= \graphLdelta$, that is,
 \begin{equation}\label{eq:graphLa+}
 	\sol \in	\underset{\bx \in X}{\argmin} \left\{ \frac{1}{2}\|K\bx - \by^\delta \|_2^2 + \alpha\|\graphLdelta\bx\|_1 \right\}.
 \end{equation}

Let us remark that the regularization term in \eqref{eq:graphLa+}  intrinsically relies on both the noise intensity $\delta$ and the observed data $\by^\delta$. This dependency introduces a layer of complexity into the analysis, making it highly non-standard.

\begin{figure}[htb]
	\centering
	\begin{tikzpicture}
    % Draw the larger ellipses
    \node[draw, ellipse, minimum width=4cm, minimum height=5.5cm, label=below:{$X$}] (X) at (0,0) {};
    \node[draw, ellipse, minimum width=4cm, minimum height=5.5cm, label=below:{$Y$}] (Y) at (6,0) {};

% Draw the curvy shape around xgt with a light orange or red color
\draw[fill=LightGrey, smooth cycle, tension=0.6] plot coordinates {
	(-1, -2) (-0.6, -2.5)   (0.3, -2)  (0.2, -1) (-1, -1)
};

    % Add elements inside the ellipses
    \node[circle, fill, inner sep=2pt, label=below:{$\gt$}] (xgt) at (-0.5,-1.8) {}; % Changed to normal math mode
    \node[circle, fill, inner sep=2pt, label=left:{$\bx_0$}] (x0) at (-0.9, -1.3) {};
    \node[circle, fill, inner sep=2pt, label=right:{$\bx_{\textnormal{sol}}$}] (xdagger) at (-0.5, -1.5) {};
    \node[circle, fill, inner sep=2pt, label=above right:{$\Psi^\delta_\Theta$}] (psi) at (-0.8, 1.8) {};
    \node[circle, fill, inner sep=2pt, label=above right:{$\sol$}] (sol) at (0.3, 0.5) {};
    
    \node[circle, fill, inner sep=2pt, label=left:{$y$}] (y) at (5.5,0) {}; % Changed to normal math mode
    \node[circle, fill, inner sep=2pt, label=right:{$y^{\delta}$}] (ydelta) at (6,1.5) {}; % Changed to normal math mode

    % Connect elements with an arrow
\draw[->, >=stealth] (xgt) to[bend right=30] node[midway, fill=white] {$K$} (y); % Curved arrow with label    
\draw[->, >=stealth] (y) -- (ydelta) node[midway, fill=white, font=\small] {$\eta$};

\draw[->, >=stealth] (ydelta) to[] node[midway, fill=white, font=\small] {$\Psi$} (psi); % Curved arrow with label
\draw[->, >=stealth] (ydelta) to[] node[pos=0.55, fill=white, font=\small] {\texttt{graphLa+$\Psi$}} (sol); % Curved arrow with label    

%\draw[->, >=stealth] (sol) -- (xdagger) node[midway, right, xshift=0.3em, fill=white, font=\small] {$\delta\to0$};
%\draw[->, >=stealth] (psi) -- (x0) node[midway, left, xshift=-0.3em, fill=white, font=\small] {$\delta\to0$};

\draw[->, >=stealth] (psi) -- +(0.2,-0.5) -- +(-0.5,-0.8) -- + (-0.5,-1.2) -- +(-0,-1.4) -- +(-1,-2) -- 
+(0.5,-1.5) -- +(-0.6,-2.4) -- +(-0,-2.6) --(x0);
\node[font=\tiny] at (-0.8, 0.75) {$\delta\to0$};

\draw[->, >=stealth, darkred] (sol) .. controls (1.3, -0.5) and (-1, -0.5) .. (xdagger);
\node[font=\tiny] at (0.8, -0.5) {\textcolor{darkred}{$\delta\to0$}};

% Define intermediate points along the first arrow
% \coordinate (A1) at ($(ydelta)!0.25!(psi)$);
% \coordinate (B1) at ($(ydelta)!0.35!(psi)$);
% \coordinate (C1) at ($(ydelta)!0.75!(psi)$);
% \coordinate (D1) at ($(ydelta)!0.85!(psi)$);

% % Define intermediate points along the second arrow
% \coordinate (A2) at ($(ydelta)!0.25!(sol)$);
% \coordinate (B2) at ($(ydelta)!0.35!(sol)$);
% \coordinate (C2) at ($(ydelta)!0.75!(sol)$);
% \coordinate (D2) at ($(ydelta)!0.85!(sol)$);

% \draw[<->, >=stealth, darkred] (A1) -- (A2);
% \draw[<->, >=stealth, darkred] (B1) -- (B2);
% \draw[<->, >=stealth, darkred] (C1) -- (C2);
% \draw[<->, >=stealth, darkred] (D1) -- (D2);

\draw[->, >=stealth, darkred] (psi) -- (sol);
\end{tikzpicture}
	\caption{A schematic representation of the \textnormal{$\texttt{graphLa+}\Psi$} method. The reconstructors $\Psi_\Theta$ do not necessarily need to be a regularization method, and this is represented by the piecewise linear path of $\Psi^\delta_\Theta$ as $\delta$ goes to $0$. However, when combined with the graph Laplacian in the Tikhonov method \eqref{model_eq2}, it generates a convergent and stable regularization operator, that is, \textnormal{$\texttt{graphLa+}\Psi$}, which is represented by the smooth red path. See \Cref{sec:model_setting,sec:regularization} for more details on the notation.}
	\label{fig:graphLa+_schema}
\end{figure}
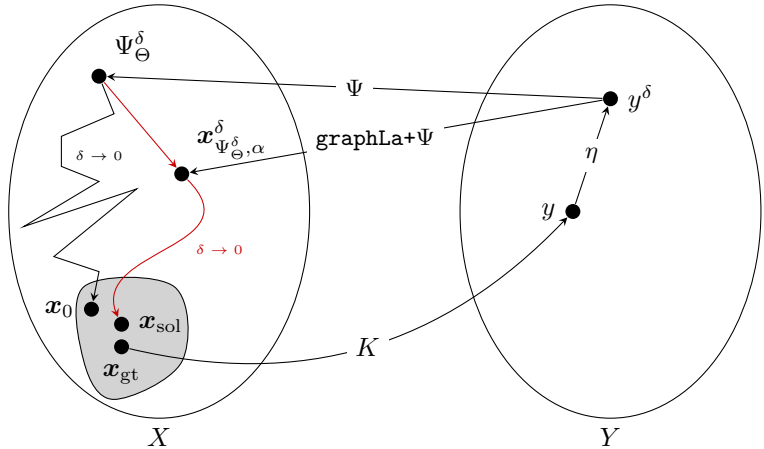

In this study, we demonstrate that under certain, albeit very weak, hypotheses, \graphLa is a convergent and stable regularization method in the limit as $\delta \to 0$. To the best of our knowledge, this is the first time that such a theoretical analysis has been carried out. 

Properties of the graph Laplacian are the key ingredients. Informally speaking, it helps to ``chain'' the reconstructors $\Psi_\Theta$ in the original and well-established regularization method \cref{model_eq2}. See  \Cref{fig:graphLa+_schema} for a visual representation of the method.

For these reasons, and due to the minimal assumptions about $\Psi_\Theta$, it becomes feasible to select reconstructors that may not be convergent regularizing methods or those whose regularization properties lack rigorous proof yet show empirical effectiveness in certain applications. Owing to the influence of the graph Laplacian, the overall \graphLa method maintains regularization and stability regardless, cfr. Section~\ref{ssec:theoretical_restults}.

In this context, a very interesting choice for $\Psi_\Theta$ is represented by the class of Deep Neural Networks (DNNs). They are highly nonlinear operators usually characterized by a paramount number of trainable parameters. Generally speaking, the set of parameters is trained by feeding the DNNs with a large number of data and minimizing a loss function. Thanks to the exponential increase of dedicated computational power, DNNs achieved state of the art performances in various applications. Particularly in recent years, they have been extensively studied in the field of inverse problems, see \cite{arridge2019solving} and references therein.    

However, applying DNNs to ill-posed inverse problems presents notable drawbacks, most importantly instability and their ``black-box" nature. Firstly, DNNs are often sensitive to data perturbations and have a tendency to produce hallucinations, i.e. false yet realistic-looking artifacts, see \cite{antun2020instabilities,colbrook2022difficulty}. Secondly, the complexity of their internal mechanisms, which involve millions of parameters and nonlinear mappings, makes them challenging to understand or explain \cite{tjoa2020survey}.

These shortcomings contribute to a general mistrust in DNNs, particularly in real-world situations where accuracy, stability, and reliability are paramount, such as in medical applications. To address these concerns, we propose the integration of the \graphLa method with a DNN, named \texttt{graphLa+Net}. This combination leverages the regularization properties of \graphLa, resulting in a stable, convergent, and mathematically robust regularization method that incorporates a DNN, effectively countering the aforementioned flaws. An initial demonstration of the potential of \texttt{graphLa+Net} is presented in \cite{bianchi2023graph}.

Our strategy with \texttt{graphLa+Net} follows other studies seeking to embed DNNs, or general learned regularizers, within a rigorous mathematical framework, see \cite{li2020nett,schwab2019deep,mukherjee2020learned,bianchi2023uniformly,alberti2021learning}. This approach is crucial for the reliable application of DNNs to ill-posed inverse problems.

While the methodology and theory we develop herein apply broadly and are not limited solely to ill-posed inverse problems in imaging, to keep the paper self-contained we will focus on 2D computerized tomography (CT) applications.

The manuscript is organized as follows.  \Cref{sec:model_setting} introduces the notation and provides preliminaries on graph and regularization theories relevant to our model setting. This section focuses on a general approach to generating a graph from an image.  \Cref{sec:regularization} presents  the \graphLa method, covering its theoretical basis. In  \Cref{sec:DNNs}, we specialize in the integration of the \graphLa method with a DNN, namely \texttt{graphLa+Net}.  \Cref{sec:numerical_experiments,sec:expsetup} showcase numerical experiments within the CT framework, employing synthetic and real data. These experiments demonstrate that the \graphLa approach notably improves the quality of approximated solutions for each method $\Psi$, with \texttt{graphLa+Net} exhibiting exceptional performance, ensuring a robust and stable implementation of DNNs for solving inverse problems. In particular, thanks to the regularization property granted by the \graphLa approach, we were able to avoid any estimate of the noise distribution during the DNN training, focusing only on the best accuracy performance of the DNN on an unperturbed dataset.

The manuscript concludes in  \Cref{sec:conclusions}, where we summarize our findings and suggest future research directions. To maintain the readability and flow of the manuscript, the technical proofs involved in  \Cref{sec:regularization} are moved to \Cref{sec:appendix}.

\section{The model setting}\label{sec:model_setting}
Hereafter, we indicate in bold any finite dimensional (column)  vector, i.e. $\bx  \in X\simeq \R^n$, and for the $p$-component of $\bx$ we write $\bx(p)$, for $p=1,\ldots, n$. We use the standard notation $\|\cdot\|_r$  for an $\ell^r$-norm, $r\in [1,\infty]$.

As a baseline assumption, we say that the unperturbed observation $\by\in Y\simeq\R^m$, which is given in \cref{eq:problem_formulation} for $\delta=0$, is the realization of the action of $K$ on an element $\gt\in X$. That is,
\begin{hypothesis}\label{hypothesis:existence}
	There exists $\gt$ such that $K\gt = \by$.
\end{hypothesis}

 On the one hand, the goal is to obtain a reliable approximation of the ground truth signal $\gt$ when we have access only to an observation $\by^\delta$ corrupted by unknown noise. On the other hand, the unperturbed system \cref{eq:Model_Equation} for $\delta=0$ can be underdetermined, the observed data $\by^\delta$ can not belong to the range of $K$, or the (pseudo) inverse of $K$ can be very sensitive even to small levels of $\delta>0$. For those reasons, we say that the inverse problem of finding a solution for \cref{eq:Model_Equation} is ill-posed, and needs to be regularized. For an overview of ill-posed inverse problems and regularization techniques, we refer the reader to \cite{engl1996regularization,hansen1998rank,scherzer2009variational}.

Our analysis will heavily depend on graph structures, and so in  \Cref{ssec:graph_theory} we provide the groundwork of Graph Theory, and in  \Cref{ssec:images_and_graphs} we explore the relationship between images and graphs. This relationship will be exploited in the numerical experiments of  \Cref{sec:numerical_experiments}.

\subsection{Graph Theory}\label{ssec:graph_theory}
Fix a finite set $P$. We indicate by $C(P)$ the set of real-valued functions on $P$, that is,
\begin{equation*}
	C(P) \coloneqq \{\bx \mid \bx : P \to \R\}.
\end{equation*}
Assuming an implicit ordering of the elements in $P$, it is evident that if the cardinality of the set $P$ is $n$ then $C(P) \simeq \mathbb{R}^n \simeq X$. Hence, we use the same notation for finite-dimensional vectors, even when referring to elements $\bx \in C(P)$. From \Cref{sec:regularization} we will often make the identifications $C(P) = \R^n= X$. 

For a detailed introduction to the graph setting as presented here, see \cite{keller2021graphs}. 	
\begin{definition}\label{def:graph}
	We define a graph over a set $P$ as a pair $G=(w,\mu)$ given by:
	\begin{itemize}
		\item A nonnegative \emph{edge-weight} function $w\colon P\times P \to [0,\infty)$ that satisfies 
		\begin{enumerate}[i)]
			\item\label{hp_graph:simmetry} Symmetry: $w(p,q)=w(q,p)$ for every $p,q \in P$;
			\item\label{hp_graph:loops} No self-loops: $w(p,p)=0$ for every $p \in P$.
		\end{enumerate}
		\item A positive \emph{node measure} $\mu \colon P \to (0,\infty)$.
	\end{itemize}
\end{definition}
The definition of graph can be relaxed, for example, allowing for non-symmetric  edge-weight functions and self-loops.

Two nodes are connected if $w(p,q)>0$, and in that case we write $p\sim q$. A finite \emph{walk} is a finite sequence of nodes $\{p_i\}_{i=0}^k$ such that $w(p_{i},p_{i+1})>0$ for $i=0,\ldots,k-1$. A subset $Q\subseteq P$ is \emph{connected} if for every pair of nodes $p,q \in Q$ there is a finite walk such that $p_0=p$, $p_k=q$, and each $p_i$ belongs to $Q$. A connected subset  $Q\subseteq P$ is a \emph{connected component} of $P$ if it is maximal with respect to the ordering of inclusion.

\begin{definition}[graph Laplacian]\label{def:graph_Laplacian}
	The \emph{graph Laplacian}  $\Delta \colon  C(P) \to C(P)$ associated to the graph $G=(w,\mu)$ is defined by the action
	\begin{align*}
		\Delta \bx(p)&\coloneqq \frac{1}{\mu(p)}\sum_{q\in P} w(p,q)\left(\bx(p) - \bx(q)\right).\label{formal_laplacian2} 
		% &= \frac{1}{\mu(p)}\left(\deg(p)\bx(p) - \sum_{q\in P} w(p,q) \bx(q)\right).\nonumber
	\end{align*}
\end{definition} 
Notice that $\Delta$ does not depend on the ordering (or labeling) of the elements in $P$. 
% \red{Moreover, $\Delta$ is positive semi-definite with respect to the inner product 
% $$
% \langle \bx, \by \rangle \coloneqq \sum_{p\in P} \bx(p)\by(p) \mu(p),
% $$
% making it attractive for several optimization algorithms.}

% \begin{figure}
% 	\centering
% 	\input{example_graph}
% 	\caption{Visual representation of a finite and connected graph with $P=\{p_1,\ldots, p_7\}$. Whenever $w(p_i, p_j) > 0$, we draw an edge between the nodes $p_i$ and $p_j$. The thickness of the edge serves to emphasize the magnitude of $w(p_i, p_j)$.
% 	}\label{fig:example_graph}
% \end{figure}

\subsubsection{Distance-based graph creation}\label{sssec:building_a_graph}
Given a finite set $P = \{p \mid p = 1,\ldots,n\}$ , fix a distance $\operatorname{dist}(\cdot,\cdot)$ on the set $P$ and a nonnegative function $h_{\textnormal{d}}\colon \R \to [0,+\infty)$ such that 	$h_{\textnormal{d}}(0)=0$.
Then it follows  that 
\begin{equation*}
    w_{\textnormal{d}}(p,q)\coloneqq h_{\textnormal{d}}(\operatorname{dist}(p,q))
\end{equation*}
is an edge-weight function on $P$, since $w_{\textnormal{d}}$ satisfies \cref{hp_graph:simmetry,hp_graph:loops} from \Cref{def:graph}. It is based on the geometric properties of $P$ induced by the distance $\operatorname{dist}(\cdot,\cdot)$. The magnitude of the connections between two nodes $p$ and $q$  is then regulated by $h_{\textnormal{d}}$.

Fix now an element $\bx\in C(P)$, another nonnegative function $h_{\textnormal{i}} \colon \R \to [0,+\infty)$ and finally define
\begin{equation}\label{def:induced_w}	
	w_{\bx}(p,q)\coloneqq  \underbrace{w_{\textnormal{d}}(p,q)}_{\text{geometry}}\cdot  \underbrace{h_{\textnormal{i}}(|\bx(p)-\bx(q)|)}_{\bx\text{ intensity}}.
\end{equation}
The edge-weight function \cref{def:induced_w} between two nodes depends on both the ``physical'' distance, thanks to $w_{\textnormal{d}}$, and  the ``variation of intensity'' of $\bx$, thanks to $h_{\textnormal{i}}(|\bx(p)-\bx(q)|)$. This dual dependence has a twofold effect: it can separate nodes that reside in different and unrelated regions of the space, and it weights the magnitude of the connections depending on the difference of intensities of $\bx$. See  \cite[Equation (4.8)]{gilboa2009nonlocal} for one of the first implementations of this idea. For an example of application, see \Cref{ssec:images_and_graphs}. 

Now choose any strictly positive function, which may depend on $\bx$,
\begin{equation}\label{def:induced_mu}
	 \mu_{\bx}(p) >0 \quad \forall \; p\in P.
\end{equation} 

Therefore, following \Cref{def:graph}, for any fixed $\bx\in C(P)$,  the pair \cref{def:induced_w}-\cref{def:induced_mu} is a graph $G$, on the set $P$, induced by $\bx$. We will write then $\Delta_{\bx}$ to indicate the associated graph Laplacian, as per \Cref{def:graph_Laplacian}.

\subsection{Images and graphs}\label{ssec:images_and_graphs}
We briefly describe here how images and graphs can be related, following the strategy and notation described in the previous  \Cref{sssec:building_a_graph}, and provide a practical example. 

First, any image is made by the union of several pixels $p \in P$ disposed on a grid of dimension $H\times W$, where $H,W \in \N$. It is then natural to identify them as ordered pairs $p=(i,j)\in \mathbb{Z}^2$ with $i=1,\ldots, H$ and $j=1,\ldots,W$. 

As a consequence, a reasonable choice to connect two pixels by their ``physical'' distance is to set 
\begin{equation*}
	\operatorname{dist}(p,q)\coloneqq \|p-q\|_1 \quad \mbox{and} \quad h_{\textnormal{d}}(t)\coloneqq \mathds{1}_{(0,R]}(t),
\end{equation*}
where  $\mathds{1}_{(0,R]}$ is the indicator function of the set $(0,R]$ and $R$ is a parameter of control that tells the maximum distance allowed for two pixels to be neighbors. If $0<\|p-q\|_1\leq R$, then $p$ and $q$ are connected with an edge of magnitude $1$, viz. $w_{\textnormal{d}}(p,q)=1$.

\begin{figure}[tb!]
	\begin{minipage}{0.25\textwidth}
		\centering
		% Colors
    \definecolor{myblue}{RGB}{0,85,164}
    \definecolor{mywhite}{RGB}{255,255,255}
    \definecolor{myorange}{RGB}{255,140,0}
    \definecolor{mydarkorange}{RGB}{204,102,0}  % Darker Orange
    \definecolor{mylightblue}{RGB}{128,170,207}  % Lighter Blue
    \definecolor{mymixed}{RGB}{178,123,49}
    
\begin{tikzpicture}[scale=0.4]

    % Initially fill all squares with myorange
    \foreach \x in {0,...,6} {
        \foreach \y in {0,...,6} {
            \fill[myorange] (\x,\y) rectangle ++(1,1);
        }
    }

    % Fill specified squares with blue
    \foreach \x in {0,...,5} {
        \foreach \y in {0,...,2} {  % Adjusted to remove bottom row
            \fill[myblue] (\x,\y) rectangle ++(1,1);
        }
    }
    \foreach \x in {6} {
        \foreach \y in {0,...,5} {
            \fill[myblue] (\x,\y) rectangle ++(1,1);
        }
    }

    % Making the interface less straight by choosing specific squares to be blue
    \fill[myblue] (5,3) rectangle ++(1,1);
    \fill[myblue] (5,5) rectangle ++(1,1);
    \fill[myblue] (3,3) rectangle ++(1,1);
    \fill[myblue] (1,3) rectangle ++(1,1);

    % Change the color of some orange squares to mydarkorange
    \fill[mydarkorange] (1,5) rectangle ++(1,1);
    \fill[mydarkorange] (3,5) rectangle ++(1,1);
    \fill[mydarkorange] (3,6) rectangle ++(1,1);

    % Change the color of some blue squares to mylightblue
    \fill[mylightblue] (0,1) rectangle ++(1,1);
    \fill[mylightblue] (1,0) rectangle ++(1,1);
    \fill[mylightblue] (0,0) rectangle ++(1,1);

    % Grid
    \foreach \x in {0,...,6} {
        \foreach \y in {0,...,6} {
            % Draw rectangle for the grid
            \draw (\x,\y) rectangle ++(1,1);
                    }
    }

\end{tikzpicture}
	\end{minipage}%
	\begin{minipage}{0.25\textwidth}
		\centering
		% Colors
    \definecolor{myblue}{RGB}{0,85,164}
    \definecolor{mywhite}{RGB}{255,255,255}
    \definecolor{myorange}{RGB}{255,140,0}
    \definecolor{mydarkorange}{RGB}{204,102,0}  % Darker Orange
    \definecolor{mylightblue}{RGB}{128,170,207}  % Lighter Blue
    \definecolor{mymixed}{RGB}{178,123,49}
    
\begin{tikzpicture}[scale=0.4]

    % Initially fill all squares with myorange
    \foreach \x in {0,...,6} {
        \foreach \y in {0,...,6} {
            \fill[myorange] (\x,\y) rectangle ++(1,1);
        }
    }

    % Fill specified squares with blue
    \foreach \x in {0,...,5} {
        \foreach \y in {0,...,2} {  % Adjusted to remove bottom row
            \fill[myblue] (\x,\y) rectangle ++(1,1);
        }
    }
    \foreach \x in {6} {
        \foreach \y in {0,...,5} {
            \fill[myblue] (\x,\y) rectangle ++(1,1);
        }
    }

    % Making the interface less straight by choosing specific squares to be blue
    \fill[myblue] (5,3) rectangle ++(1,1);
    \fill[myblue] (5,5) rectangle ++(1,1);
    \fill[myblue] (3,3) rectangle ++(1,1);
    \fill[myblue] (1,3) rectangle ++(1,1);

    % Change the color of some orange squares to mydarkorange
    \fill[mydarkorange] (1,5) rectangle ++(1,1);
    \fill[mydarkorange] (3,5) rectangle ++(1,1);
    \fill[mydarkorange] (3,6) rectangle ++(1,1);

    % Change the color of some blue squares to mylightblue
    \fill[mylightblue] (0,1) rectangle ++(1,1);
    \fill[mylightblue] (1,0) rectangle ++(1,1);
    \fill[mylightblue] (0,0) rectangle ++(1,1);

    % Grid
    \foreach \x in {0,...,6} {
        \foreach \y in {0,...,6} {
            % Draw rectangle for the grid
            \draw (\x,\y) rectangle ++(1,1);
            
            % black circle at the center of each square
            \fill[black] (\x+0.5,\y+0.5) circle (0.12);
        }
    }

\end{tikzpicture} % Assuming this is the name of your second file
	\end{minipage}%
	\begin{minipage}{0.25\textwidth}
		\centering
		 % Colors
    \definecolor{myblue}{RGB}{0,85,164}
    \definecolor{mywhite}{RGB}{255,255,255}
    \definecolor{myorange}{RGB}{255,140,0}
    \definecolor{mydarkorange}{RGB}{204,102,0}  % Darker Orange
    \definecolor{mylightblue}{RGB}{128,170,207}  % Lighter Blue
    \definecolor{mymixed}{RGB}{178,123,49}

\begin{tikzpicture}[scale=0.4]

    % Initially fill all squares with myorange
    \foreach \x in {0,...,6} {
        \foreach \y in {0,...,6} {
            \fill[myorange] (\x,\y) rectangle ++(1,1);
        }
    }

    % Fill specified squares with blue
    \foreach \x in {0,...,5} {
        \foreach \y in {0,...,2} {  % Adjusted to remove bottom row
            \fill[myblue] (\x,\y) rectangle ++(1,1);
        }
    }
    \foreach \x in {6} {
        \foreach \y in {0,...,5} {
            \fill[myblue] (\x,\y) rectangle ++(1,1);
        }
    }

    % Making the interface less straight by choosing specific squares to be blue
    \fill[myblue] (5,3) rectangle ++(1,1);
    \fill[myblue] (5,5) rectangle ++(1,1);
    \fill[myblue] (3,3) rectangle ++(1,1);
    \fill[myblue] (1,3) rectangle ++(1,1);

    % Change the color of some orange squares to mydarkorange
    \fill[mydarkorange] (1,5) rectangle ++(1,1);
    \fill[mydarkorange] (3,5) rectangle ++(1,1);
    \fill[mydarkorange] (3,6) rectangle ++(1,1);

    % Change the color of some blue squares to mylightblue
    \fill[mylightblue] (0,1) rectangle ++(1,1);
    \fill[mylightblue] (1,0) rectangle ++(1,1);
    \fill[mylightblue] (0,0) rectangle ++(1,1);

    % Grid
    \foreach \x in {0,...,6} {
        \foreach \y in {0,...,6} {
            % Draw rectangle for the grid
            \draw (\x,\y) rectangle ++(1,1);
            
        }
    }

  % Connect circles vertically in white
    \foreach \x in {0,...,6} {
        \foreach \y in {0,...,5} {
            \draw[line width=1mm, mywhite] (\x+0.5,\y+0.5) -- (\x+0.5,\y+1.5);
        }
    }

 % Connect circles horizontally in white
    \foreach \y in {0,...,6} {
        \foreach \x in {0,...,5} {
            \draw[line width=1mm, mywhite] (\x+0.5,\y+0.5) -- (\x+1.5,\y+0.5);
        }
    }

    % Grid
    \foreach \x in {0,...,6} {
        \foreach \y in {0,...,6} {
            
            % black circle at the center of each square
            \fill[black] (\x+0.5,\y+0.5) circle (0.12);
        }
    }
    
\end{tikzpicture} % Assuming this is the name of your third file
	\end{minipage}%
	\begin{minipage}{0.25\textwidth}
		\centering
		% Colors
    \definecolor{myblue}{RGB}{0,85,164}
    \definecolor{mywhite}{RGB}{255,255,255}
    \definecolor{myorange}{RGB}{255,140,0}
    \definecolor{mydarkorange}{RGB}{204,102,0}  % Darker Orange
    \definecolor{mylightblue}{RGB}{128,170,207}  % Lighter Blue
    \definecolor{mymixed}{RGB}{178,123,49}

\begin{tikzpicture}[scale=0.4]

    % Initially fill all squares with myorange
    \foreach \x in {0,...,6} {
        \foreach \y in {0,...,6} {
            \fill[myorange] (\x,\y) rectangle ++(1,1);
        }
    }

    % Fill specified squares with blue
    \foreach \x in {0,...,5} {
        \foreach \y in {0,...,2} {  % Adjusted to remove bottom row
            \fill[myblue] (\x,\y) rectangle ++(1,1);
        }
    }
    \foreach \x in {6} {
        \foreach \y in {0,...,5} {
            \fill[myblue] (\x,\y) rectangle ++(1,1);
        }
    }

    % Making the interface less straight by choosing specific squares to be blue
    \fill[myblue] (5,3) rectangle ++(1,1);
    \fill[myblue] (5,5) rectangle ++(1,1);
    \fill[myblue] (3,3) rectangle ++(1,1);
    \fill[myblue] (1,3) rectangle ++(1,1);

    % Change the color of some orange squares to mydarkorange
    \fill[mydarkorange] (1,5) rectangle ++(1,1);
    \fill[mydarkorange] (3,5) rectangle ++(1,1);
    \fill[mydarkorange] (3,6) rectangle ++(1,1);

    % Change the color of some blue squares to mylightblue
    \fill[mylightblue] (0,1) rectangle ++(1,1);
    \fill[mylightblue] (1,0) rectangle ++(1,1);
    \fill[mylightblue] (0,0) rectangle ++(1,1);

    % Grid
    \foreach \x in {0,...,6} {
        \foreach \y in {0,...,6} {
            % Draw rectangle for the grid
            \draw (\x,\y) rectangle ++(1,1);
            
        }
    }
    
    % Connect circles vertically in white
    \foreach \x in {0,...,6} {
        \foreach \y in {0,...,5} {
            \draw[line width=0.05mm, mywhite] (\x+0.5,\y+0.5) -- (\x+0.5,\y+1.5);
        }
    }

     % Connect circles vertically in white thick lines
    \foreach \x in {2,...,6} {
        \foreach \y in {0} {
            \draw[line width=1mm, mywhite] (\x+0.5,\y+0.5) -- (\x+0.5,\y+1.5);
        }
    }

    \foreach \x in {1} {
        \foreach \y in {1,...,2} {
            \draw[line width=1mm, mywhite] (\x+0.5,\y+0.5) -- (\x+0.5,\y+1.5);
        }
    }

    \foreach \x in {3} {
        \foreach \y in {1,...,2} {
            \draw[line width=1mm, mywhite] (\x+0.5,\y+0.5) -- (\x+0.5,\y+1.5);
        }
    }

    \foreach \x in {5} {
        \foreach \y in {1,...,2} {
            \draw[line width=1mm, mywhite] (\x+0.5,\y+0.5) -- (\x+0.5,\y+1.5);
        }
    }

    \foreach \x in {6} {
        \foreach \y in {1,...,4} {
            \draw[line width=1mm, mywhite] (\x+0.5,\y+0.5) -- (\x+0.5,\y+1.5);
        }
    }

    \draw[line width=1mm, mywhite] (4+0.5,1+0.5) -- (4+0.5,1+1.5);

    \foreach \x in {0} {
        \foreach \y in {3,...,5} {
            \draw[line width=1mm, mywhite] (\x+0.5,\y+0.5) -- (\x+0.5,\y+1.5);
        }
    }

    \foreach \x in {2} {
        \foreach \y in {3,...,5} {
            \draw[line width=1mm, mywhite] (\x+0.5,\y+0.5) -- (\x+0.5,\y+1.5);
        }
    }

    \foreach \x in {2} {
        \foreach \y in {3,...,5} {
            \draw[line width=1mm, mywhite] (\x+0.5,\y+0.5) -- (\x+0.5,\y+1.5);
        }
    }

    \foreach \x in {4} {
        \foreach \y in {4,...,5} {
            \draw[line width=1mm, mywhite] (\x+0.5,\y+0.5) -- (\x+0.5,\y+1.5);
        }
    }

    \draw[line width=1mm, mywhite] (3+0.5,5+0.5) -- (3+0.5,5+1.5); 

    \draw[line width=1mm, mywhite] (0+0.5,0+0.5) -- (0+0.5,0+1.5); 
    \draw[line width=1mm, mywhite] (2+0.5,1+0.5) -- (2+0.5,1+1.5); 
    \draw[line width=1mm, mywhite] (4+0.5,3+0.5) -- (4+0.5,3+1.5);

    % Connect circles horizontally in white
    \foreach \y in {0,...,6} {
        \foreach \x in {0,...,5} {
            \draw[line width=0.05mm, mywhite] (\x+0.5,\y+0.5) -- (\x+1.5,\y+0.5);
        }
    }

     % Connect circles horizontally in white thick lines
    \foreach \y in {0,...,1} {
        \foreach \x in {2,...,5} {
            \draw[line width=1mm, mywhite] (\x+0.5,\y+0.5) -- (\x+1.5,\y+0.5);
        }
    }

    \foreach \y in {2} {
        \foreach \x in {3,...,5} {
            \draw[line width=1mm, mywhite] (\x+0.5,\y+0.5) -- (\x+1.5,\y+0.5);
        }
    }

    \foreach \y in {4} {
        \foreach \x in {0,...,4} {
            \draw[line width=1mm, mywhite] (\x+0.5,\y+0.5) -- (\x+1.5,\y+0.5);
        }
    }

    \draw[line width=1mm, mywhite] (0+0.5,0+0.5) -- (0+1.5,0+0.5);
    \draw[line width=1mm, mywhite] (0+0.5,0+0.5) -- (0+1.5,0+0.5);
    \draw[line width=1mm, mywhite] (4+0.5,6+0.5) -- (4+1.5,6+0.5);
    \draw[line width=1mm, mywhite] (5+0.5,5+0.5) -- (5+1.5,5+0.5);
    \draw[line width=1mm, mywhite] (5+0.5,3+0.5) -- (5+1.5,3+0.5);
    \draw[line width=1mm, mywhite] (0+0.5,2+0.5) -- (0+1.5,2+0.5);
    \draw[line width=1mm, mywhite] (1+0.5,1+0.5) -- (1+1.5,1+0.5);
    \draw[line width=1mm, mywhite] (0+0.5,6+0.5) -- (1+1.5,6+0.5);
    \draw[line width=1mm, mywhite] (1+0.5,2+0.5) -- (2+1.5,2+0.5);
    \draw[line width=1mm, mywhite] (5+0.5,6+0.5) -- (5+1.5,6+0.5);
 % Connect circles horizontally in white half-thick lines
    \draw[line width=0.4mm, mywhite] (1+0.5,0+0.5) -- (1+1.5,0+0.5);
    \draw[line width=0.4mm, mywhite] (0+0.5,1+0.5) -- (0+1.5,1+0.5);
    \draw[line width=0.4mm, mywhite] (0+0.5,5+0.5) -- (3+1.5,5+0.5);
    \draw[line width=0.4mm, mywhite] (2+0.5,6+0.5) -- (3+1.5,6+0.5);

 % Connect circles horizontally in white half-thick lines
    \draw[line width=0.4mm, mywhite] (1+0.5,0+0.5) -- (1+1.5,0+0.5);
    \draw[line width=0.4mm, mywhite] (0+0.5,1+0.5) -- (0+1.5,1+0.5);
    \draw[line width=0.4mm, mywhite] (0+0.5,5+0.5) -- (3+1.5,5+0.5);
    \draw[line width=0.4mm, mywhite] (2+0.5,6+0.5) -- (3+1.5,6+0.5);
 % Connect circles vertically in white half-thick lines
    \draw[line width=0.4mm, mywhite] (1+0.5,0+0.5) -- (1+0.5,0+1.5);
    \draw[line width=0.4mm, mywhite] (0+0.5,1+0.5) -- (0+0.5,1+1.5);
    \draw[line width=0.4mm, mywhite] (1+0.5,4+0.5) -- (1+0.5,5+1.5);
    \draw[line width=0.4mm, mywhite] (3+0.5,4+0.5) -- (3+0.5,4+1.5);

  % Grid
    \foreach \x in {0,...,6} {
        \foreach \y in {0,...,6} {
            
            % black circle at the center of each square
            \fill[black] (\x+0.5,\y+0.5) circle (0.12);
        }
    }
    
\end{tikzpicture} % Assuming this is the name of your fourth file
	\end{minipage}\caption{Simple outline of how to build a graph from an image $\bx$. To be read from left to right. Left: a $7\times7$ pixels image made by orange-like and blue-like square pixels. The color intensity of each pixel is given by the pixel-wise evaluation of a function $\bx$. Center-left: each pixel corresponds to one node, represented by a black circle. Since the pixels are disposed on a grid, each node can be associated to an ordered pair in $\mathbb{Z}^2$. Center-right: the geometric edge-weight function $w_{\textnormal{d}}$ in \cref{def:induced_w} is given by $\mathds{1}_{(0,1]}(\|p-q\|_1)$, that is, two nodes $p, q$ are connected if and only if $\|p-q\|_1 =1$, and in that case the magnitude of the connection is one. Right: the magnitude of an edge between two nodes is then weighted by $h_{\textnormal{i}}(\|\bx(p)-\bx(q)\|)\in (0,1]$, where $h_{\textnormal{i}}(t) = \exp\{-t^2/\sigma^2\}$ is the Gaussian function. The role of $h_{\textnormal{i}}$ is to measure the difference of intensity between two adjacent pixels, and it is close to zero when two pixels have very different color intensities. This is represented by the different thicknesses of the edges connecting two adjacent pixels, where a thick edge means a very similar color intensity and a thin edge means a very different color intensity.}\label{fig:example_image2graph}
\end{figure}
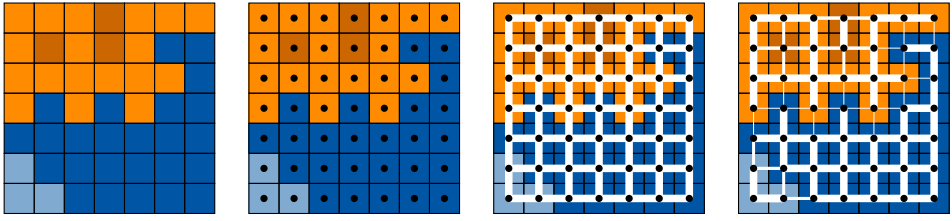

Second, a gray-scale image is given by the light intensities of its pixels. That is, a gray-scale image can be represented by a function $\bx \in C(P)$ such that $\bx(p) \in [0,1]$, where $0$ means black and $1$ means white. A common choice to weight the connection of two different pixels by their light intensities is to use the Gaussian function, that is
\begin{equation*}
	h_{\textnormal{i}}(t) \coloneqq \rm{e}^{-\frac{t^2}{\sigma^2}}, \qquad \sigma>0.
\end{equation*}
The reason lies in the relationship between the heat kernel and the discretization of the Laplacian on a manifold. See \cite{bronstein2017geometric} for discussions about the Laplacian on discretized manifolds, and  \cite{calatroni2017graph,gilboa2009nonlocal,lou2010image} for  applications of the Gaussian function to define the weights of the graph Laplacian. 

We arrive then at the definition of the edge-weight function in \eqref{def:induced_w}, applied to our case, that is
\begin{equation*}
	w_{\bx}(p,q) = \mathds{1}_{(0,R]}(\|p-q\|_1)\rm{e}^{-\frac{|\bx(p)-\bx(q)|^2}{\sigma^2}}.
\end{equation*}
See \Cref{ssec:hp_validity} for more details. The values of $\sigma$ only modify the shape of the Gaussian weight function: small values of $\sigma$ correspond to a tight distribution, while larger values result in wider curves. Conversely, \(R\) influences the sparsity of the graph Laplacian operator. Specifically, a larger \(R\) leads to a denser matrix, which implies that each matrix-vector product requires more computations. For these reasons, we choose \(R \leq 5\) to keep the computational cost low at each iteration, and $\sigma \leq 10^{-3}$ to avoid strong connections between uncorrelated pixels that are sufficiently close to each other. These choices are used, for instance, in \cite{aleotti2024fractional}.

For colored images, it is a little bit different. For example, in the RGB representation, the color of a pixel is given by the combination of the light intensities of three channels R(ed), G(green), and B(lue). Therefore, in principle a colored image should be regarded as a function $\bx : P \to \R^3$, where $\bx(p) = (\bx_{R} (p), \bx_{G}(p), \bx_{B}(p))$ is a vector-valued function whose elements represent the light intensities for each channel. However, it is common to assume that the ill-posed operator $K$ acts independently on each channel, and therefore the regularization is made on each channel separately. If for some reason this assumption can not be made, then we can simply modify the definition of \eqref{def:induced_w} in the following way:
\begin{equation*}
	w_{\bx}(p,q) \coloneqq  w_{\textnormal{d}}(p,q)\cdot h_{\textnormal{i}}(\|\bx(p)-\bx(q)\|),
\end{equation*}
where $\|\cdot\|$ is any appropriate norm in $\R^3$. See Figure \ref{fig:example_image2graph} for a simple example of building a graph from an image.

\section{The \graphLa regularization method}\label{sec:regularization}
In this section, we introduce the \graphLa method and we show that it is a regularization method. 

Let us begin by considering a family of operators $\{\Psi_\Theta : Y \rightarrow X\}$ which we generally refer to as \emph{reconstructors}. $\Psi_\Theta$ is  not necessarily linear, and $\Theta \in \R^k$ denotes  its parameters. This family of reconstructors is at the core of our \graphLa method. It has the very important role of giving a first approximation of $\gt$, since upon this approximation the reconstruction step of our method is built. 

We need to enforce some (weak) regularity on $\Psi_\Theta$. 

\begin{hypothesis}\label{hypothesis:Psi_convergence}
There exists an element $\bx_0 \in X$ and a parameter choice rule $\Theta = \Theta(\delta, \by^\delta)$  such that 
\begin{equation*}
		\|\Psi_{\Theta(\delta,\by^\delta)} (\by^\delta) - \bx_0\|_2 \to 0 \qquad \mbox{as } \delta\to 0.
	\end{equation*}
\end{hypothesis}
The above \Cref{hypothesis:Psi_convergence} will guarantee that the \graphLa method, which we will introduce properly in \cref{graphModel}, is a convergent regularization method.  The next hypothesis will guarantee stability for \cref{graphModel}.
\begin{hypothesis}\label{hyp:stability}
Let  $\by^\delta$ and $\Theta\coloneqq \Theta(\delta,\by^\delta)$ be fixed, and $\{\delta_k\}$ and $\{\by^{\delta_k}\}$ be sequences such that $\delta_k\to\delta$ and $\by^{\delta_k}\rightarrow \by^\delta$ for $k\rightarrow \infty$. Writing $\Theta_k\coloneqq \Theta(\delta_k, \by^{\delta_k})$, then
\begin{equation*}
    \Theta_k\to\Theta \quad \mbox{and} \quad  \|\Psi_{\Theta_k}(\by^{\delta_k})-\Psi_{\Theta}(\by^{\delta})\|_2 \to 0 \; \mbox{for } k\to \infty.
\end{equation*}
\end{hypothesis}
In the next two examples, we make clear that the above assumptions are pretty weak and they can be satisfied by several large classes of reconstructors. In particular, the pair $(\Psi_\Theta, \Theta)$ does not need to be a convergent regularization method. 
\begin{example}\label{ex:Lipchitz}
	A simple example of a family of reconstructors that satisfies \Cref{hypothesis:Psi_convergence} is the case when we identify them with a single,  (locally) Lipschitz continuous operator. That is, fix $\Theta \equiv \hat{\Theta}$ for every $\delta$ and $\by^\delta$, and choose an operator $\Psi_{\hat{\Theta}}$ such that $\|\Psi_{\hat{\Theta}}(\by_1) -\Psi_{\hat{\Theta}}(\by_2)\|_2 \leq L \|\by_1 - \by_2 \|_2$. Thanks to \cref{eq:problem_formulation} and the Lipschitz condition, \Cref{hypothesis:Psi_convergence} is then  verified with $\bx_0 \coloneqq \Psi_{\hat{\Theta}}(\by)$. In the same way, \Cref{hyp:stability} is verified again by the Lipschitz property of $\Psi_{\hat{\Theta}}$ and because $\hat{\Theta}$ is fixed and independent of $k$.
\end{example}
The situation of the preceding example will occur later, when $\Psi_{\hat{\Theta}}$ is implemented as a trained DNN, as discussed in \Cref{sec:DNNs} and \Cref{ssec:DNN_training}.  Let us emphasize that DNNs are typically non-regularizing methods. See for example \cite{tan2024provably,eliasof2023drip,gottschling2020troublesome}, and references therein.
\begin{example}\label{ex:regularization_methods}
	A less trivial family of reconstructors that satisfy \Cref{hypothesis:Psi_convergence} is the one given by any typical regularization operators, as per \cite[Definition 3.1]{engl1996regularization}, for example. Let $\Theta \in (0,\infty)$ and $\Psi_\Theta$ be continuous, and  $K^\dagger$ be the usual Moore-Penrose pseudo-inverse of $K$. Fix $\bx_0 \coloneqq K^\dagger \by$ and observe that $K\bx_0 = \by$, since $\by\in \operatorname{range}(K)$ by \Cref{hypothesis:existence}.  By definition of regularization operator, there exists a parameter choice rule $\Theta = \Theta(\delta, \by^\delta)$ such that
 \begin{equation*}
     \sup\{\|\Psi_{\Theta(\delta,\by^\delta)}(\by^\delta) - \bx_0  \|_2 \mid \by^\delta\in Y,\; \|\by^\delta - K\bx_0\|_2\leq \delta \}\to 0 \quad \mbox{as } \delta \to 0,
 \end{equation*}
and \Cref{hypothesis:Psi_convergence} is then verified. About \Cref{hyp:stability},  this is a bit more involved and depends on the regularization method itself and the parameter choice rule. We will show a specific example in \Cref{example:Tikhonov}. 
\end{example}

The class of convergent regularization methods which fits \Cref{ex:regularization_methods} is vast, and not necessarily restricted to the Hilbert setting. We mention Tikhonov-Phillips methods (both ordinary and iterative variants) \cite{hochstenbach2011fractional,klann2008regularization,bianchi2015iterated,bianchi2017generalized}, popular $\ell^p$-$\ell^q$ methods with general regularization term \cite{lanza2015generalized,chung2019flexible},  and framelets \cite{cai2008framelet,huang2013nonstationary,cai2016regularization}.

As detailed in \Cref{sssec:building_a_graph}, we can define a graph Laplacian induced by 
$$
\Psi_\Theta^\delta\coloneqq \Psi_\Theta(\by^\delta)\in X\simeq C(P),
$$
and we will denote it by $\graphLdelta$. Consider then the following Tikhonov-type variational method for \cref{eq:Model_Equation},
\begin{equation}\label{graphModel}
	\sol \in  \underset{\bx \in X}{\argmin} \left\{\frac{1}{2}\|K\bx - \by^\delta\|_2^2 + \alpha \|\graphLdelta\bx\|_1\right\}.
\end{equation}
We call \cref{graphModel} the \graphLa method. 

Let us focus our attention on the regularization term $\mathcal{R}(\bx, \by^\delta)\coloneqq \|\graphLdelta\bx\|_1$. As a first remark, notice that, unlike typical regularization methods,  $\mathcal{R}$ depends not only on $\bx$ but also on the data $\by^\delta$. This complicates any attempt at studying the convergence of \cref{graphModel} for $\delta \to 0$.

As a second remark, indicating with $w_{\Psi_\Theta^\delta}$ the edge-weight function \cref{def:induced_w}  with $\bx$ replaced by $\Psi_\Theta^\delta$, we have that minimizing $\mathcal{R}$   means to force $\sol$ to be constant on the regions where $w_{\Psi_\Theta^\delta}(p, q)$  is ``large''. This is a direct consequence of \Cref{def:graph_Laplacian}, 
\begin{equation*}\label{eq:Lap_inequality}
	|\graphLdelta \bx(p)| = \frac{1}{\mu_{\Psi_\Theta^\delta}(p)}\left|\sum_{q\in P} w_{\Psi_\Theta^\delta}(p, q) (\bx(p) - \bx(q))\right|.
\end{equation*}

For example, setting $w_{\Psi_\Theta^\delta}$ as in \cref{eq:edge-weight-function:applications}, then $w_{\Psi_\Theta^\delta}(p,q)$ attains its maximum value of $1$ if and only if $p\sim q$ and $\Psi_\Theta^\delta(p)=\Psi_\Theta^\delta(q)$. Therefore, $\| \graphLdelta \bx\|_1\approx 0$ if $\bx$ is constant on the regions where $\Psi_\Theta^\delta$ is constant.

Put simply, we can say that $\mathcal{R}$ helps to distinguish the regions of uniformity from the regions of interfaces.  In intuition, once $\Psi_\Theta^\delta$ reconstructs the region of interfaces in $\gt$, the final solution $\sol$ will be a good approximation of the ground truth $\gt$, even though $\Psi_\Theta^\delta$ deviates from $\gt$ in other aspects.

As a last comment, the $\ell^1$-norm in $\mathcal{R}$ is introduced to enforce sparsity and preserve discontinuities on the approximated solution $\sol$. This is mainly in view of the imaging applications of Section \ref{sec:numerical_experiments}. All the theory developed here works with the $\ell^1$-norm replaced by any $\ell^r$-norm for $r>1$.

% \begin{algorithm}
% 	\caption{\graphLa}
% 	\label{alg:graphLa+}
% 	\begin{algorithmic}[1]
% 		\STATE{Inputs: $\by^\delta$, $K$, $\Psi_{\Theta}$, $\alpha$;}
% 		\STATE{Output: $\sol$;}
% 		\STATE{Compute $\Psi^\delta_{\Theta}= \Psi_{\Theta}(\by^\delta)$;}
% 		\STATE{Compute $w_{\Psi^\delta_{\Theta}}$ and $\mu_{\Psi^\delta_{\Theta}}$ as in \cref{def:induced_w,def:induced_mu}, respectively;}
% 		\STATE{Build $\Delta_{\Psi^\delta_\Theta}$ as per \Cref{def:graph_Laplacian};}
% 		\STATE{Compute a minimizer $\sol$ of \cref{graphModel};}
% 		\RETURN $\sol$.
% 	\end{algorithmic}
% \end{algorithm}

\subsection{Convergence and stability results}\label{ssec:theoretical_restults}
In this subsection, we study the regularizing properties of \graphLa. Since the proofs involved are quite technical, to help the readability of this subsection we move them to \Cref{sec:appendix}.

First, we need a definition of solution for \cref{eq:Model_Equation}.  Under \Cref{hypothesis:Psi_convergence}, let $\bx_0$ and $\Theta=\Theta(\delta,\by^\delta)$ be such that
\begin{equation}\label{eq:limit_solution}
\bx_0\coloneqq \lim_{\delta \to 0} \Psi_{\Theta(\delta,\by^\delta)} (\by^\delta).
\end{equation}
\begin{definition}\label{def:graph-minimizing_solution}
We call $\bx_{\textnormal{sol}}$ a \emph{graph-minimizing solution} with respect to $\bx_0$, defined in \cref{eq:limit_solution},	if $K\bx_{\textnormal{sol}} = \by$ and 
	\begin{equation}
		\|\graphL \bx_{\textnormal{sol}}\|_1 = \min\{  \|\graphL\bx\|_1 \mid \bx \in X, \; K\bx = \by \}.
	\end{equation}
\end{definition}

\begin{remark}
A graph-minimizing solution $\bx_{\textnormal{sol}}$ is a  pre-image of $\by$ which minimizes the functional $\mathcal{R}(\bx) = \|\graphL \bx\|_1$. If the operator $K$ is injective, then there exists one and only one graph-minimizing solution and $\bx_{\textnormal{sol}} = \gt$, thanks to \Cref{hypothesis:existence}. However, in general, $\gt$ is not necessarily a graph-minimizing solution when $K$ is not injective. Loosely speaking, a graph-minimizing solution is an approximation of the (inaccessible) ground-truth with respect to some a-posteriori information encoded in~$\bx_0$, as per \Cref{eq:limit_solution}. \Cref{def:graph-minimizing_solution} is a special instance of the definition of $\mathcal{R}$-minimizing solutions; see for example~\cite[Definition 3.24]{scherzer2009variational}.
\end{remark}

Let us indicate with $w_{\Psi_\Theta^\delta}$  and $w_0$ the edge-weight functions in \cref{def:induced_w} induced by $\Psi_\Theta^\delta$ and $\bx_0$, respectively. To make our analysis work, we need three last hypotheses which are related to properties of $w_{\Psi_\Theta^\delta}$ and $\graphLdelta$. As we will discuss in \Cref{ssec:hp_validity}, those hypotheses are easy to be satisfied. 

\begin{hypothesis}\label{hypothesis:null_stability}
	For every $p,q \in P$, $w_{\Psi_\Theta^\delta}>0$ if and only if $w_0(p,q)>0$.
\end{hypothesis}
The following lemma is an immediate consequence of the above hypothesis.
\begin{lemma}\label{lem:null_stability}
	Under \Cref{hypothesis:null_stability}, there is an invariant subspace $V\subseteq C(P)$ such that $\ker(\graphLdelta)=\ker(\graphL)=V$  for every $\Psi_\Theta^\delta$.
\end{lemma}
\begin{proof}
See \Cref{sec:appendix}. 
\end{proof}
The invariant subspace $V$  replaces $\ker(D)$ in the typical null-space condition $\ker(K)\cap \ker(D) = \{\mathbf{0}\}$, invoked for functionals of the form  \cref{model_eq2}.

\begin{hypothesis}\label{hypothesis:null_intersection}
	$\ker(K)\cap V=\{\boldsymbol{0}\}$.
\end{hypothesis}

\vspace{0.1cm}
\noindent The last assumption we need is on the function $h_{\textnormal{i}}$ in \cref{def:induced_w} and the node measure $\mu_\bx$  in \cref{def:induced_mu}.

\begin{hypothesis}\label{hypothesis:differentiability}
The function	$h_{\textnormal{i}}$ and  the map $\bx \mapsto \mu_{\bx}(p)$, for every fixed $p\in P$, are  Lipschitz.
\end{hypothesis}
\vspace{0.1cm}

\noindent The above hypotheses are not difficult to check in practice. In  \Cref{ssec:hp_validity} we  provide specific choices of \cref{def:induced_w,def:induced_mu}, for our numerical experiments, and we  show that \Cref{hypothesis:null_stability,hypothesis:null_intersection,hypothesis:differentiability} are satisfied.

Throughout the remainder of this section, we will assume the validity of all the hypotheses introduced thus far: \Cref{hypothesis:existence,hypothesis:differentiability,hypothesis:null_intersection,hypothesis:Psi_convergence,hypothesis:null_stability,hyp:stability}. 

The first three results are about the existence and uniqueness of the graph-minimizing solution and the well-posedness of \cref{graphModel}.

\begin{proposition}\label{prop:existence_min}
There exists a graph-minimizing solution $\bx_{\textnormal{sol}}$.
\end{proposition}
\begin{proof}
See \Cref{sec:appendix}. 
\end{proof}

\begin{proposition}\label{prop:well-posedness}
	For every fixed $\delta, \alpha >0$ and $\by^\delta \in Y$,  there exists a  solution $\sol$ for the variational problem \cref{graphModel}.
\end{proposition}
\begin{proof}
    See \Cref{sec:appendix}.
\end{proof}
\begin{corollary}\label{cor:uniqueness}
	If $K$ is injective, then $\bx_{\textnormal{sol}}$ and $\sol$ are unique.
\end{corollary}	
\begin{proof}
   See \Cref{sec:appendix}. 
\end{proof}
\begin{remark}
Without the injectivity property, uniqueness can fail. The main culprit is the $\ell^1$-norm in the regularization term. However, it is possible to achieve uniqueness in a less stringent manner, namely, for every $\by^\delta$ outside a set of negligible measures. The approach should be in line with \cite{candes2013simple,ali2019generalized}, but adapted to this specific context. That being said, relaxing the assumptions to regain the uniqueness of the solutions falls beyond the scope of the current work.
\end{remark}
The next theorems provide convergence and stability results.

\begin{theorem}\label{thm:convergence}
	Assume that $\alpha \colon (0,+\infty) \to (0,+\infty)$ satisfies 
	\begin{subequations}
		\begin{equation}\label{eq:thm:convergence:a}
			\lim_{\delta \to 0} \alpha(\delta) =0, 
		\end{equation}	
		\begin{equation}\label{eq:thm:convergence:b}
			\lim_{\delta \to 0} \frac{\delta^2}{\alpha(\delta)}=0.
		\end{equation}
	\end{subequations}
	
	Fix a sequence $\{\delta_k\}$ such that 
	\begin{equation*}
		\lim_{k\to\infty} \delta_k = 0, \qquad \|\by^{\delta_k}-\by\|_2 \leq \delta_k, 
	\end{equation*}
and set $\alpha_k \coloneqq \alpha(\delta_k)$. Let $\Theta=\Theta(\delta,\by^\delta)$ be the parameter choice rule as in \Cref{hypothesis:Psi_convergence}, and set $\Theta_k \coloneqq \Theta(\delta_k, \by^{\delta_k})$. Then every sequence $\{\bx_k\}$ of elements that minimize the functional  \cref{graphModel}, with $\delta_k$ and $\Theta_k$, has a convergent subsequence. The limit $\bx_{\textnormal{sol}}$ of the convergent subsequence $\{\bx_{k'}\}$is a graph-minimizing solution with respect to $\bx_0$, and 
	\begin{equation*}
		\lim_{k'}\|\Delta_{\Psi^{\delta_{k'}}_{\Theta_{k'}}}\bx_{k'}\|_1 = \|\graphL\bx_{\textnormal{sol}}\|_1.
	\end{equation*}
	If $\bx_{\textnormal{sol}}$ is unique, then $\bx_k \to \bx_{\textnormal{sol}}$. 
\end{theorem}
\begin{proof}
   See \Cref{sec:appendix}.  
\end{proof}

\begin{theorem}\label{thm:stability}
Let now $\by^\delta$ be fixed and $\{\delta_k\}$ and $\{\by^{\delta_k}\}$ be sequences such that $\delta_k\to\delta$ and $\by^{\delta_k}\to\by^\delta$ for $k\to\infty$.	 Then every sequence $\{\bx_k\}$ with
	\begin{equation*}
		\bx_k \in  \underset{\bx \in X}{\argmin} \left\{\frac{1}{2}\|K\bx - \by^{\delta_k}\|_2^2 + \alpha \|\Delta_{\Psi^{\delta_k}_{\Theta_k}}\bx\|_1\right\},
	\end{equation*}
	has a converging subsequence $\{\bx_{k^{'}}\}$ such that
    \begin{equation*}
        \lim_{k^{'}}\bx_{k^{'}}\in\underset{\bx \in X}{\argmin}\left\{\frac{1}{2}\|K\bx-\by^{\delta}\|_2^2+\alpha\|\Delta_{\Psi^{\delta}_\Theta}\bx\|_1\right\}.
    \end{equation*}
\end{theorem}
\begin{proof}
	See \Cref{sec:appendix}.  
\end{proof}

\section{$\texttt{graphLa+Net}$}\label{sec:DNNs}
% DEEP NEURAL NETWORKS
Among the possible choices for reconstructors $\Psi_\Theta$, we propose to use a Deep Neural Network (DNN). In this case, the set of parameters $\Theta$ contains matrices and vectors, which are the building blocks of DNNs. For a mathematical introduction to neural networks, we refer to \cite{bishop2024deep}. Informally speaking, a DNN is a long chain of compositions of affine operators and nonlinear activation functions. The set of parameters $\Theta$ is then trained by minimizing a loss function over a large number of data, see \Cref{ssec:DNN_training}. When considering $\Psi_\Theta$ as a DNN, we specify it by calling the method $\texttt{graphLa+Net}$. A first overview of $\texttt{graphLa+Net}$ was proposed in \cite{bianchi2023graph}. 

Note that, in principle, it is possible to make the network parameters $\Theta$ dependent on the noise level $\delta$ by training it multiple times for different values of $\delta$. However, this is rarely done in practice due to the significant amount of time and energy consumption it would require. This limitation has always been a crucial challenge in employing DNNs for regularizing ill-posed problems, as it necessitates the estimate of an optimal noise level $\delta$ which is suitable for diverse applications. Additionally, opting for $\delta=0$ is generally not a good choice due to the typical high sensitivity of DNNs to the noise, as observed in \cite{morotti2021green,antun2020instabilities,evangelista2023ambiguity}. 

Nonetheless, the regularizing property of \graphLa effectively addresses this issue. In the subsequent discussion, we will consider DNNs as reconstructors with a fixed $\hat{\Theta}$, where $\hat{\Theta}$ has been trained over a \textit{noiseless} dataset.  As will be shown in  \Cref{ssec:COULE,ssec:Mayo}, the resulting $\texttt{graphLa+Net}$ is not only regularizing and stable but also significantly superior in performance, despite the inherent instability of the original DNN.

\subsection{The architecture}
In our DNN model, we utilize a modified version of the U-net, as detailed in \cite{ronneberger2015u}. The structure of this modified U-net is illustrated in \Cref{fig:resunet_architecture}. U-net is a widely recognized multi-scale Convolutional Neural Network architecture, known for its effectiveness in processing images with global artifacts. This fully convolutional network features a symmetrical encoder-decoder structure, employing strided convolutions to expand its receptive field. The encoder layers' strides create distinct levels of resolution within the network. Each level comprises a set number of blocks, where a block consists of a convolutional layer with a fixed number of channels, followed by batch normalization and a ReLU activation function. The number of convolutional channels is doubled at each successive level, starting from a baseline number in the first layer. Specifically, our network is designed with four levels and a baseline of 64 convolutional channels.

As mentioned, the decoder mirrors the encoder but uses upsampling convolutional layers in place of strided convolutions. Furthermore, to preserve high-frequency details, skip connections link the final layer of each encoder level to the corresponding first layer of the decoder.

The network we consider in this paper is called Residual U-net (ResU-net) and it was proposed in \cite{evangelista2023rising}. This network is a variation of the original U-net architecture, modified through two key changes. First, we have reconfigured the skip connections to work as additions rather than concatenations, a strategy aimed at reducing the total number of parameters. Second, we introduce a residual connection that links the input and output layers directly, which implies that the network learns the residual mapping between the input and the expected output. The importance of the residual connection has been observed in  \cite{han2016deep}, where the authors proved that the residual manifold containing the artifacts is easier to learn than the true image manifold. 

After the training procedure, the resulting DNN is Lipschitz continuous. Therefore, \Cref{hypothesis:Psi_convergence} is verified, as noted in \Cref{ex:Lipchitz}. In general, the Lipschitz constant $L$ depends on various factors and can be large, potentially affecting the uniform convergence of $\graphLdelta$ in \Cref{lem:stability}. However, in the numerical experiments reported in \Cref{sec:numerical_experiments}, no issues were observed, and the overall stability of the \texttt{graphLa+Net} method was very good. If necessary, it is possible to uniformly bound $L$ from above by 1 through different training strategies; see, for example, \cite{miyato2018spectral}.

\begin{figure}[ht]
   \centering
	\includegraphics[width=\textwidth]{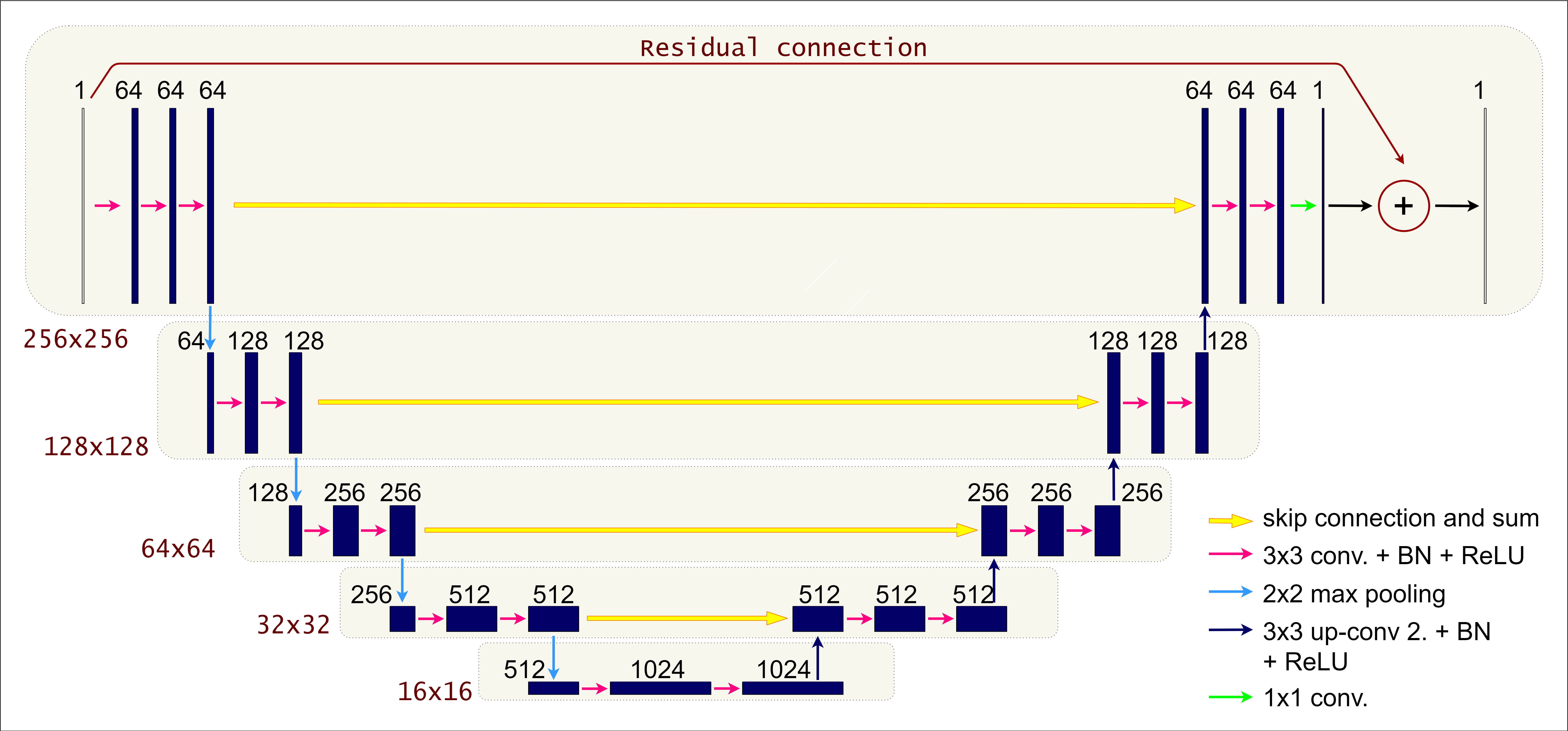}
    \caption{A diagram of the ResU-net architecture used in the experiments.}
    \label{fig:resunet_architecture}
\end{figure}

\section{Experimental setup}\label{sec:expsetup}
In this section, we describe the setting and the setup of the \graphLa method in our numerical experiments.

\subsection {Sparse view Computed Tomography}\label{ssec:CT}
X-rays Computed Tomography (CT) is a widespread imaging system particularly useful for detecting injuries, tumors, internal bleeding, bone fractures, and various medical conditions. It is essentially constituted by a source rotating around an object and emitting X-ray beams from a fixed number of angles along its arc trajectory. Passing through the interior of the object, a quantity of radiation proportional to the density of the tissues is absorbed, and the resulting rays are measured by a detector. The collection of all measurements at different projection angles is called the sinogram.

In this case, the matrix $K \in \R^{m \times n}$ represents the discrete Radon operator, $\bx \in X=\R^n$ is the  2D object discretized in $n$ pixels (voxels) and vectorized, and $\by^{\delta} \in Y=\R^m$ is the vectorized sinogram. Vectorization is simply performed by considering the one-dimensional lexicographic ordering of pixels $p=(i,j)$, which in practice is done by following the column-major or row-major order. For example, if the original 2D image is given by $H\times W$ pixels, then $n = H\cdot W$.

Here, the dimension $m$  depends on the number of projection angles $n_a$.
In medical applications, sparse view CT denotes the case in which only a small number of angles $n_a$ is considered, to reduce the radiation dose absorbed by the patient. In real-world applications, what is called full acquisition in CT has two or three projections per degree; acquisitions that reduce the projections by a factor of 3 or 4 are considered sparse acquisitions. In this particular setting, the matrix $K$ is under-determined (i.e. $m<n$). To faithfully simulate the real medical scenario in our experiments we model $K$ as the fan beam Radon transform, modeling the spread of X-rays with a fan trajectory.

\subsection{The datasets and the test problems}
We evaluate the \graphLa algorithms using two distinct image datasets. The first is the COULE dataset, which comprises synthetic images with a resolution of 256x256 pixels. These images feature ellipses and lines of varying gray intensities against a dark background. This dataset is publicly available on Kaggle \cite{coule-dataset}. The second dataset is a subsampled version of the AAPM Low Dose CT Grand Challenge dataset, provided by the Mayo Clinic \cite{moen2021low}. It contains real chest CT image acquisitions, each also at a resolution of 256x256 pixels.

To simulate the sinogram $\by^\delta$ we consider $n_a$ different angles evenly distributed within the closed interval $[0, 179]$, where $n_a = 60$ in the experiments with COULE dataset and $n_a = 180$ in the experiments with Mayo dataset. The sinograms are generated using the IRtools toolbox \cite{gazzola2019ir}.

\begin{figure}\label{fig:COULE_start}
    \centering
    \subfigure[True image $\gt$]{\includegraphics[width= 4cm, height = 4cm]{COULE_true}}    
    \subfigure[Sinogram]{\includegraphics[width= 4 cm, height = 4 cm]{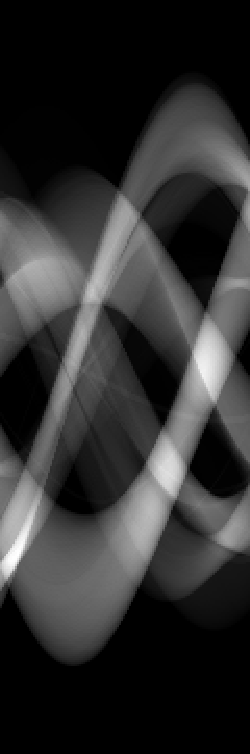}}
    \caption{(a): Example of a $\gt$ image from COULE dataset. (b): The resulting sinogram}
\end{figure}

\begin{figure}\label{fig:Mayo_start}
    \centering
    \subfigure[True image $\gt$]{\includegraphics[width=4cm, height=4cm]{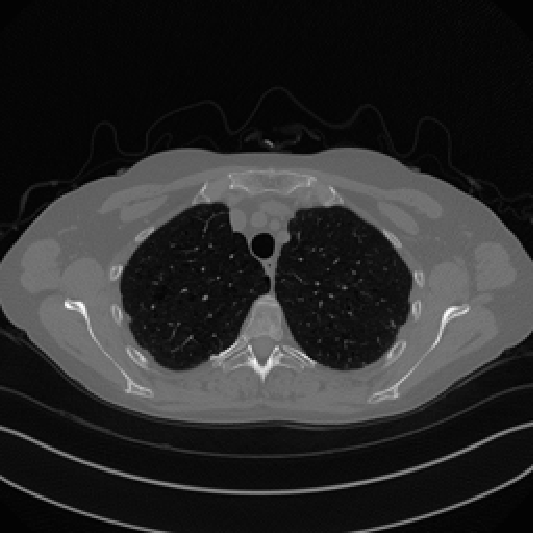}}    
    \subfigure[Sinogram]{\includegraphics[width=4cm, height=4cm]{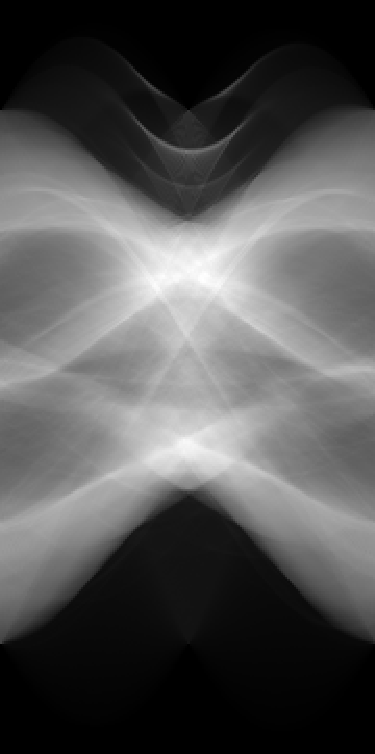}}
    \caption{(a): Example of a $\gt$  image from Mayo dataset. (b): The resulting sinogram}
\end{figure}

In Figures \ref{fig:COULE_start} and \ref{fig:Mayo_start}, we present an example of the true image and its resulting sinogram of size $n_d \times n_a$, where $n_d = \lfloor \sqrt{2n} \rfloor$ is the number of pixels of the detector. To simulate real-world conditions, we add white Gaussian noise $\boldsymbol{\xi}$ to the sinogram at an intensity level of $\delta$, indicating that the norm of the noise is $\delta$ times the norm of the sinogram. In particular, we compute $\by^\delta$ as:
\begin{equation*}
    \by^\delta = \by + \delta \| \by \|\frac{\boldsymbol{\xi}}{\| \boldsymbol{\xi}\|}.
\end{equation*}

\subsection{Neural Network training}\label{ssec:DNN_training}
For the DNN introduced in \Cref{sec:DNNs}, we randomly selected 400 pairs of images from  COULE and 3,305 pairs of images from Mayo as training sets, all of the form $(\gt, \by^\delta)$.
Following the discussion in \Cref{sec:DNNs},  we train the DNN on the training sets in a supervised manner without extra noise, specifically by setting $\delta=0$.  The process involves finding $\hat{\Theta}$ that minimizes the Mean Squared Error (MSE) between the predicted reconstruction $\Psi_{\Theta}(\by)$ and the ground-truth solution $\gt$, viz. $\operatorname{MSE}=\frac{1}{n} \| \Psi_\Theta(\by)- \gt \|_2^2$. Once the training process is completed, we set $\Theta \equiv \hat{\Theta}$. We did not apply any regularization technique in the training phase. 

Since the  ResU-net architecture is fully convolutional, the input required by the network is an image. Hence, the input sinogram $\by$ has to be pre-processed through a fast algorithm mapping the sinogram to a coarse reconstructed image, such as FBP  \cite{morotti2021green, jin2017deep} or a few iterations of a regularizing algorithm \cite{morotti2023increasing, evangelista2023rising}. For those experiments, we chose the FBP.

\begin{figure}\label{fig:loss_function}
    \centering  \subfigure{\includegraphics[width=0.6\linewidth]{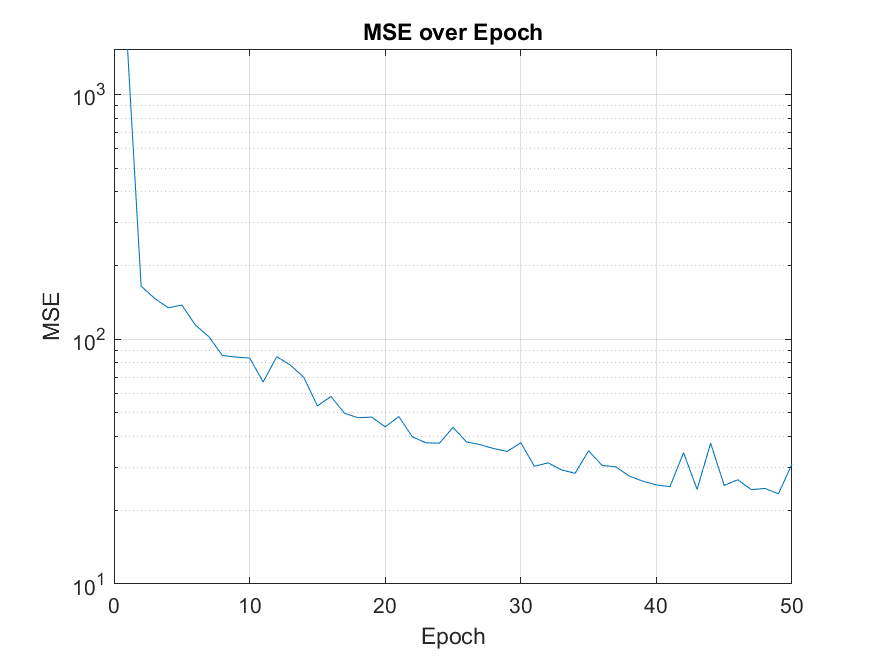}} 
    \caption{Training loss plot for COULE experiment.}
\end{figure}
The networks have been trained on an NVIDIA RTX A4000 GPU card with 16Gb of VRAM, for a total of 50 epochs and a batch size of 10, arresting it after the loss function stopped decreasing as shown in Figure \ref{fig:loss_function}. We used Adam optimizer with a learning rate of 0.001, $\beta_1 = 0.9$, and $\beta_2 = 0.999$, in all the experiments.

\subsection{Construction of the graph Laplacian}\label{ssec:hp_validity}

In the numerical experiments, to compute the graph Laplacian we consider the following edge-weight function $\omega_\bx$  and  node function $\mu_\bx$:
\begin{equation}\label{eq:edge-weight-function:applications}
	w_{\bx}(p,q) = \mathds{1}_{(0,R]}(\|p-q\|_{\infty}){\rm{e}}^{-\frac{|\bx(p)-\bx(q)|^2}{\sigma^2}}, \quad \mu_\bx(p) = \mu_\bx \coloneqq   \sqrt{\sum_{p,q\in P} w_{\bx}^2(p,q)}.
\end{equation}
The theory developed in \Cref{sec:regularization} relies on several hypotheses. In \Cref{prop:validity1}, we show that with those choices, $\omega_\bx$ grants straightforwardly the validity of  \Cref{hypothesis:null_stability,hypothesis:null_intersection}, and in \Cref{prop:validity2} we show that both $\omega_\bx$  and $\mu_\bx$ satisfy \Cref{hypothesis:differentiability}. As a corollary, we have that the constant $c$ that appears in \Cref{lem:stability} is independent from the dimension $n$ of the vector space $C(P)\simeq X = \R^n$.

\begin{proposition}\label{prop:validity1}
The edge-weight function $\omega_\bx$ ensures the validity of \Cref{hypothesis:null_stability}, and of  \Cref{hypothesis:null_intersection} when $K$ represents the discrete Radon operators as introduced in \Cref{ssec:CT}.
\end{proposition}
\begin{proof}
Using the same notation as in \cref{def:induced_w,eq:edge-weight-function:applications}, since $h_{\textnormal{i}}(t) = \textrm{e}^{-\frac{t^2}{\sigma^2}} > 0$ for every $t$, then $\omega_\bx(p,q)>0$ if and only if $w_{\textnormal{d}}(p,q)>0$. From the choice \cref{eq:edge-weight-function:applications}, $w_{\textnormal{d}}(p,q) = \mathds{1}_{(0,R]}(\|p-q\|_1)$ and it is independent of $\bx$. This proves \Cref{hypothesis:null_stability}.

It is immediate to check that the whole set of pixels $P$ is connected with respect to $w_{\textnormal{d}}$, and therefore is connected with respect to $w_\bx$, for any $\bx$. As a consequence, indicating with $V$ the invariant subspace in \Cref{lem:null_stability}, it holds that $V = \ker(\Delta_\bx) = \{ t \boldsymbol{1} \mid t \in \R\}$ for any $\bx$, where $\boldsymbol{1}\in C(P)$ is the  constant function $\boldsymbol{1}(p) = 1$ for every $p\in P$. For a proof, see \cite[Lemmas 0.29 and 0.31]{keller2021graphs}.

Since $K$ is a discrete Radon operator, then $K\boldsymbol{1} > \boldsymbol{0}$ for any possible configuration of $K$, such as the number of angles $n_a$ or the number of pixels $n_d$ of the detector. This is due to the fact that each row of $K$ represents line integrals. See, for example, \cite[Chapter 1.4]{scherzer2009variational}. In particular, this means $\ker(K) \cap V = \{\boldsymbol{0}\}$, which is \Cref{hypothesis:null_intersection}.
\end{proof}

\begin{proposition}\label{prop:validity2}
	$\omega_\bx$  and $\mu_\bx$ satisfy \Cref{hypothesis:differentiability}.
\end{proposition}
\begin{proof}
We need to show that $h_{\textnormal{i}}$ and $\bx(p) \mapsto \mu_\bx(p)$ in \cref{eq:edge-weight-function:applications} are Lipschitz. The first part is trivial, since  $h_{\textnormal{i}}(t) = \textrm{e}^{-\frac{t^2}{\sigma^2}}$  is a smooth function with a bounded derivative for every $\sigma^2>0$.

Let us observe now that $\mu_\bx(p) = \| W_\bx\|_F$ for every $p\in P$, where $W_\bx$ is the adjacency matrix associated to $w_\bx$ and $\|\cdot\|_F$ is the Frobenius norm.  Therefore, for any $\bx, \by \in C(P)$ and any $p\in P$ it holds that
\begin{equation}\label{eq:prop:validity2-1}
|\mu_\bx(p) - \mu_\by (p)| = | \|W_\bx\|_F - \|W_\by\|_F| \leq \|W_\bx - W_\by\|_F.
\end{equation}
A generic element of  $W_\bx - W_\by$ in position $(p,q)$ is given by 
$$
w_\bx (p,q) - w_\by (p,q) = w_{\textnormal{d}}(p,q)\left(h_{\textnormal{i}}(|\bx(p) - \bx(q)|) - h_{\textnormal{i}}(|\by(p) - \by(q)|)  \right).
$$
Using the same arguments as in \Cref{eq:lem:stability2,eq:lem:stability:second_bound} of \Cref{lem:stability}, it is possible to show that 
$$
|w_\bx (p,q) - w_\by (p,q)| \leq 2w_{\textnormal{d}}(p,q)L'\|\bx - \by \|_2,
$$
where $L'$ is the Lipschitz constant of $h_{\textnormal{i}}$. From \cref{eq:edge-weight-function:applications}, it holds that 
\begin{eqnarray*}
\overline{d}=\max_{p,q}\{w_{\textnormal{d}}(p,q)\} &=& 1,\\
\max_{q\in P}\left\{\operatorname{card}\{ p\in P \mid w_{\textnormal{d}}(p,q) \neq 0  \}\right\} &=& \max_{q\in P}\left\{\operatorname{card}\{p\in P \mid \|p-q\|_1\leq R\}-1\right\} = \overline{\kappa},
\end{eqnarray*}
where $\overline{\kappa} = (2R+1)^2$. Therefore, it holds
\begin{align*}
\|W_\bx - W_\by\|_F = \sqrt{\sum_{q\in P} \sum_{p\in P}|w_\bx (p,q) - w_\by (p,q)|^2} \leq 2L'\overline{d}\overline{\kappa} \sqrt{n}\|\bx - \by \|_2,
\end{align*}
and, from \Cref{eq:prop:validity2-1}, we conclude that $\bx \mapsto \mu_\bx(p)$ is Lipschitz with constant $L''= 2L'\overline{d}\overline{\kappa} \sqrt{n}$, for every $p\in P$.
\end{proof}

% versione vecchia
% where $\overline{\kappa} = \sum_{k=0}^{\min\{2;R\}} \binom{2}{k}\binom{R}{k}2^k -1$. Therefore,
% \begin{align*}
% \|W_\bx - W_\by\|_F = \sqrt{\sum_{q\in P} \sum_{p\in P}|w_\bx (p,q) - w_\by (p,q)|^2} \leq 2L'\overline{d}\overline{\kappa} \sqrt{n}\|\bx - \by \|_2,
% \end{align*}
% and from \Cref{eq:prop:validity2-1} we conclude that $\bx \mapsto \mu_\bx(p)$ is Lipschitz with constant $L''= 2L'\overline{d}\overline{\kappa} \sqrt{n}$, for every $p\in P$.

\begin{corollary}
With the choices in \cref{eq:edge-weight-function:applications}, the constant $c$ in \Cref{lem:stability} is independent of $n$.
\end{corollary}
\begin{proof}
Using the same notation as in the proof of \Cref{lem:stability}, it is an almost straightforward application of \Cref{prop:validity2}. Indeed, with the choices in \cref{eq:edge-weight-function:applications} we have that
\begin{equation*}
\overline{d} = 1; \quad \overline{\kappa} = (2R+1)^2, \quad \overline{h}_{\textnormal{i}}=1,\quad \overline{L}'' = 2L'\overline{\kappa} \sqrt{n},
\end{equation*}
where $L'$ is the Lipschitz  constant of $h_{\textnormal{i}}$. Recalling that $\bx(p)\in[0,1]$, then $\inf_p\{\mu_\bx (p)\} \geq ne^{-\sigma^{-2}}$  and $\sup_p\{\mu_\bx (p)\} \leq n$ for every $\bx$, and from \cref{eq:lem:stability5} we can fix 
$$
c =\frac{2\overline{\kappa}\left( 2nL' + 2L'\overline{\kappa} \sqrt{n}\right)}{n^2e^{-2\sigma^{-2}}},
$$
which is uniformly bounded with respect to $n$.
\end{proof}

Let us remark that with the choice of $\mu_\bx$ in \cref{eq:edge-weight-function:applications}, even if the Lipschitz constant of $\bx \mapsto \mu_\bx(p)$ increases with $n$, the convergence of $\graphLdelta$ for $\delta\to 0$ is uniform with respect to~$n$, thanks to the preceding corollary and \Cref{lem:stability}. This reflects the observations made in \cite{bianchi2021graph}, where it was introduced  the node measure in \cref{eq:edge-weight-function:applications} to uniformly bound the spectrum of $I + \Delta_\bx^T\Delta_\bx$, with respect to the dimension $n$, and guarantee then a fast convergence of the \texttt{lsqr} algorithm.

\subsection{Optimization method}
To find an approximate solution to our $\ell^2-\ell^1$ model \cref{graphModel}, we use a Majorization–Minimization strategy combined with a Generalized Krylov Subspace approach as proposed in \cite{lanza2015generalized}. Our implementation incorporates a restarting strategy for the Krylov subspace and an automatic estimation of the regularization parameter $\alpha$ using the discrepancy principle. For further details, see \cite{buccini2023limited}.

\section{Numerical experiments}\label{sec:numerical_experiments}
% Numerical experiments

This section is divided into two parts, with each part concentrating on tests for a specific dataset. In both scenarios, we rigorously examined the performance of the $\texttt{graphLa+}\Psi$ method to provide a comprehensive analysis of the method's robustness and adaptability. 

More in detail, we consider a wide range of reconstructors $\Psi$, including Filter Back Projection (\texttt{graphLa+FBP}), standard Tikhonov (\texttt{graphLa+Tik}), Total Variation (\texttt{graphLa+TV}), and the trained DNN (\texttt{graphLa+Net}) described in the previous \Cref{sec:DNNs}.  In all cases, the ground truth images $\gt$ are drawn outside of the training sets defined for the DNN in \Cref{ssec:DNN_training}. 

The regularization parameters in the Tikhonov and TV methods are calculated using Generalized Cross Validation and the discrepancy principle, respectively.
 
For comparison, we will include the reconstruction achieved by our method using the ground truth image $\gt$ as a first approximation, labeled as $\texttt{graphLa+}\gt$. This serves as an upper bound reference for the effectiveness of all the \graphLa methods.

The quantitative results of our experiments will be measured by the reconstruction relative  error (RRE) and peak signal-to-noise ratio (PSNR), where
$$
\operatorname{RRE}(\bx)\coloneqq \frac{\|\gt-\bx\|^2}{\|\gt\|^2}, \qquad \operatorname{PSNR}(\bx)\coloneqq20\log_{10}\left(\frac{255}{\|\gt -\bx\|}\right), 
$$
and by the structural similarity index (SSIM) \cite{wang2004image}.
Finally, all the numerical tests are replicable and the codes can be downloaded from \cite{devangelista2023graphlaplus}.

\subsection{Example 1: COULE}\label{ssec:COULE}

In this first example, we tested our proposal on an image of the COULE test set acquired by $n_a = 60$ projections, a detector shape of $n_d = \lfloor 256\sqrt{2} \rfloor$, and corrupted with white Gaussian noise with level intensity of~$2\%$. Note that, since $n = 256^2 = 65536$ and $m = n_d \cdot n_a = 21720$, then $m \ll n$, meaning that the problem is highly sparse. 
Regarding the parameter selection for the edge-weight function in  \cref{eq:edge-weight-function:applications}, we chose $R=5$ and $\sigma=10^{-3}$.

\begin{table}\label{tab:Coule}
\centering
\begin{tabular}{|l|c|c|c|}
\hline
 \textbf{Initial reconstructors $\Psi$}&  \textbf{RRE} & \textbf{SSIM} & \textbf{PSNR} \\ \hline
 \texttt{FBP} & 0.1215 & 0.1220 & 18.3101    \\ 
 \texttt{Tik} & 0.0622 & 0.3280 & 24.1306  \\ 
 \texttt{TV}  & 0.0450 & 0.6793 & 26.9320  \\ 
 \texttt{Net} & 0.0205 & 0.9396 & 33.7714  \\ \hline
 \textbf{\graphLa} & \multicolumn{3}{c}{} \vline \\ \hline
 $\texttt{graphLa+FBP}$  & 0.0364 & 0.6419 & 28.7701 \\ 
 $\texttt{graphLa+Tik}$  & 0.0352 & 0.8874 & 29.0812 \\ 
 $\texttt{graphLa+TV}$   & 0.0228 & 0.9697 & 32.8313 \\ 
 $\texttt{graphLa+Net}$  & \textbf{0.0156} & \textbf{0.9724} & \textbf{36.1128}  \\ 
 \hdashline
 $\texttt{graphLa+}\gt$ & 0.0063 & 0.9820 & 43.9905 \\ 
 \hline
 \textbf{Other comparison methods} & \multicolumn{3}{c}{} \vline \\ \hline
 \texttt{ISTA} & 0.1769 & 0.8682 & 25.3177 \\ \texttt{FISTA} & 0.1776 & 0.8883 & 25.2839 \\   \texttt{NETT} & 0.0302 & 0.7531 & 30.4040  \\ \hline
\end{tabular}
\caption{Quality of initial and final reconstruction for different $\Psi$ for the COULE dataset.}
\end{table}

%The quality of the reconstructions achieved with different operators $\Psi$ is presented in Table \ref{tab:Coule}. Specifically, the upper part of the table displays the values of the initial reconstructions while the lower part displays the values obtained by \graphLa combined with the corresponding initial reconstructor $\Psi$.

The quality of the reconstructions achieved with different operators $\Psi$ is presented in Table \ref{tab:Coule}. The upper part of the table displays the values of the initial reconstructors $\Psi$, while the middle part shows the values obtained by combining \graphLa with the corresponding initial reconstructor $\Psi$. To provide a more comprehensive and complete analysis of the performance of our proposal, we also tested three other solvers. 
The resulting reconstructions are shown in Figure \ref{fig:ISTA-FISTA}. 
Given that the ground truth is sparse, we considered an $\ell^2-\ell^1$ problem with a Haar wavelet regularization operator and applied both FISTA and ISTA. Since both algorithms heavily depend on the choice of a proper step length, we computed it by estimating the Lipschitz constant of the operator $K^TK$ using ten iterations of the power method. 

\begin{figure}[htbp]
    \centering
    \subfigure[True image $\gt$]{\includegraphics[width=0.20\textwidth]{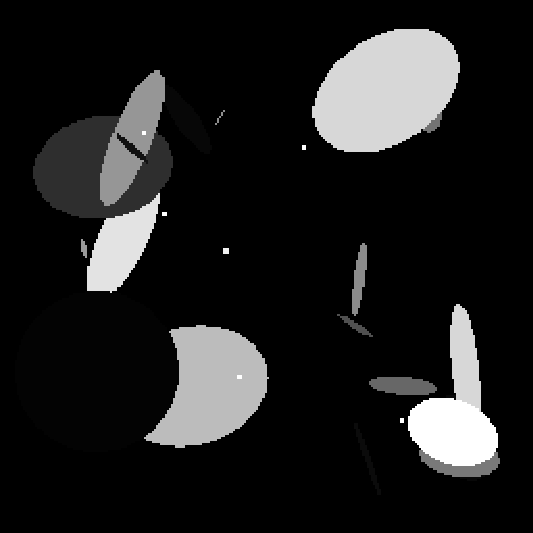}}
    \hspace{0.1cm}
    \subfigure[NETT]{\includegraphics[width=0.20\textwidth]{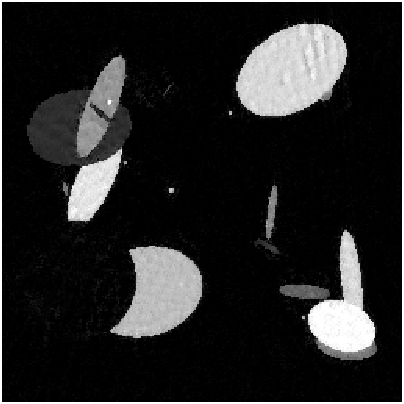}}
     \hspace{0.1cm}
    \subfigure[FISTA]{\includegraphics[width=0.20\textwidth]{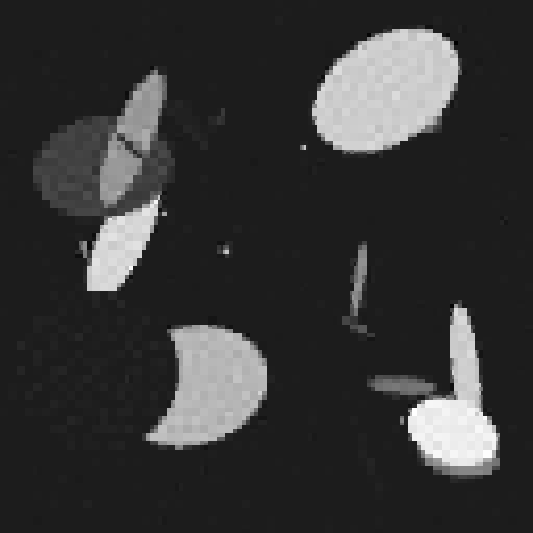}} 
     \hspace{0.1cm}
    \subfigure[ISTA]{\includegraphics[width=0.20\textwidth]{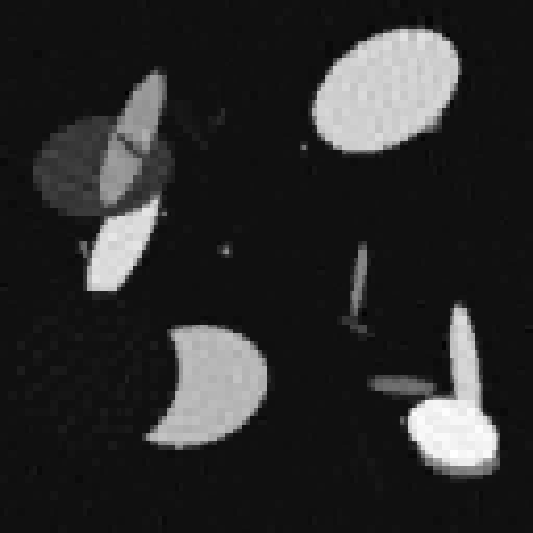}}
    \caption{Visual comparison: ground truth (a), NETT (b), FISTA (c) and ISTA (d).}\label{fig:ISTA-FISTA}
\end{figure}

Lastly, we considered a regularization DNN-based method called \texttt{NETT}~\cite{li2020nett}, where the convolutional DNN is trained to kill the artifacts. For more details about the DNN architecture and the training process we refer to \cite[Subsection 5.2]{bianchi2023uniformly}. The main regularization parameter is tuned by hand in order to get the possible best outcome. The NETT method achieved better results than the proximal methods in terms of both RRE and PSNR. However, the middle part of Table~6.1 shows that \texttt{graphLa+TV} and \texttt{graphLa+Net} are the optimal choices.

Notably, the \graphLa method results in a substantially greater improvement across all metrics for all initial reconstructors $\Psi$, with the highest performance attained by \texttt{graphLa+Net}.

As further confirmation, Figure \ref{fig:COULE_rec} displays the reconstructions obtained for different $\Psi$. The level of details and sharpness in the  \texttt{graphLa+Net} image is incomparable with all the other methods, besides \texttt{graphLa+TV}, that achieve similar performance.

\begin{figure}\label{fig:COULE_rec}
    \centering
    \subfigure[\texttt{FBP}]
    {        
        \includegraphics[width=0.20\textwidth] {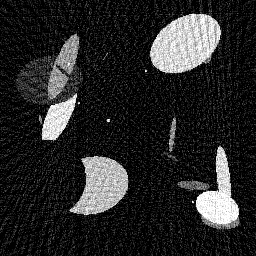} 
    }
    \subfigure[\texttt{Tik}]
    {        
        \includegraphics[width=0.20\textwidth] {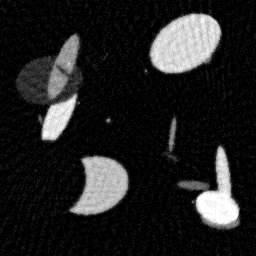} 
    }
    \subfigure[\texttt{TV}]
    {        
        \includegraphics[width=0.20\textwidth] {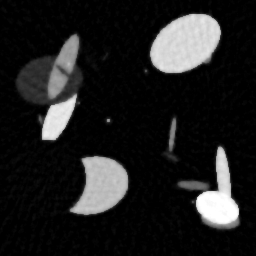} 
    }
    \subfigure[\texttt{Net}]
    {        
        \includegraphics[width=0.20\textwidth] {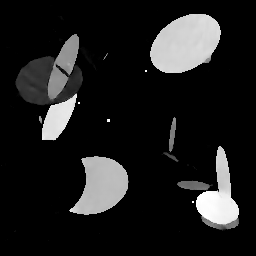} 
    }
    \\
    \subfigure[$\texttt{graphLa+FBP}$]
    {        
        \includegraphics[width=0.20\textwidth] {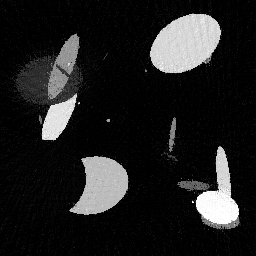} 
    }
    \subfigure[$\texttt{graphLa+Tik}$]
    {        
        \includegraphics[width=0.20\textwidth] {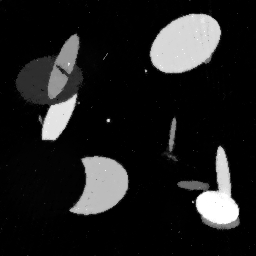} 
    }
    \subfigure[$\texttt{graphLa+TV}$]
    {        
        \includegraphics[width=0.20\textwidth] {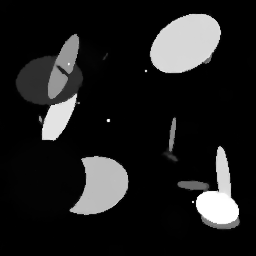} 
    }
    \subfigure[$\texttt{graphLa+Net}$]
    {        
        \includegraphics[width=0.20\textwidth] {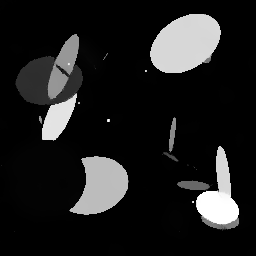} 
    }
    \caption{Initial and final reconstructions using \textnormal{$\texttt{graphLa+}\Psi$} method for different $\Psi$.}
\end{figure}

As a final investigation into the capabilities of our proposal, we tested the stability of the method against varying noise intensity. In Figure \ref{fig:COULE_stability}, we present the PSNR and SSIM values for various levels of noise. Similar to the previous analysis, the purple line represents the reconstruction obtained by utilizing the true image $\gt$ to compute the graph Laplacian. Notably, even though our neural network was trained with a $0\%$ noise level, the $\texttt{graphLa+Net}$ method consistently outperforms all other cases across all noise levels. Additionally, Figure \ref{fig:COULE_stability} shows that the integration of the graph Laplacian with the DNN serves as an effective regularization method, in contrast to the standalone application of the DNN. Indeed, while the accuracy of the  DNN does not improve as the noise intensity approaches zero, pairing it with the graph Laplacian results in a significant accuracy improvement, consistent with the effects of a regularization method.

\begin{figure}\label{fig:COULE_stability}
    \centering
    \subfigure[PSNR]
    {        
        \includegraphics[width=0.45\textwidth] {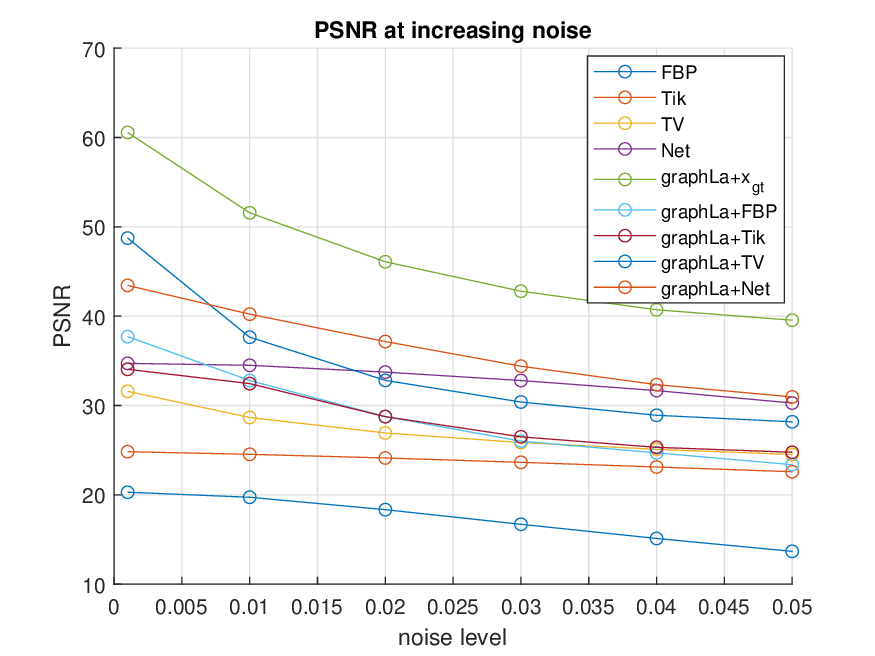} 
    }
    \subfigure[SSIM]
    {        
        \includegraphics[width=0.45 \textwidth] {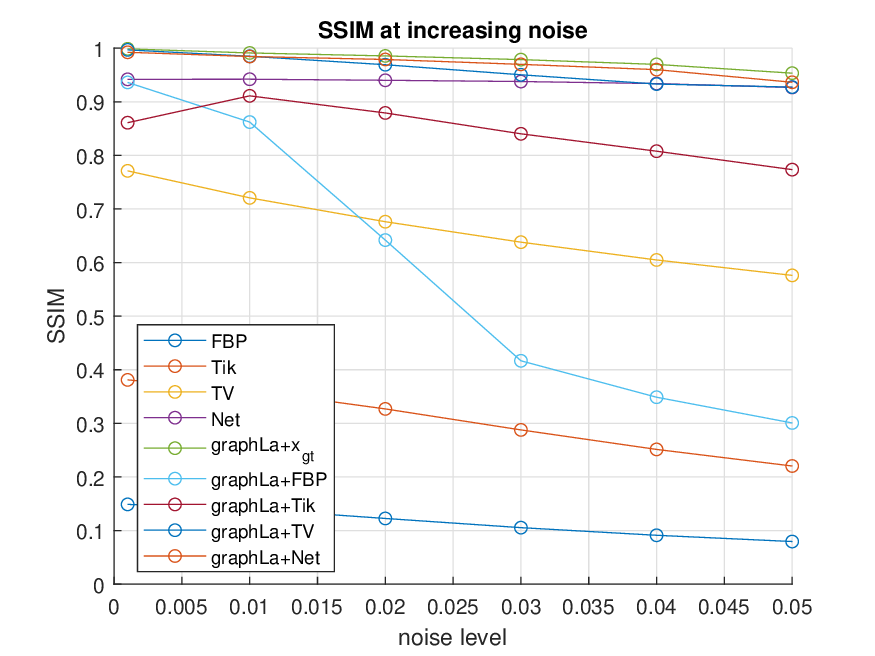} 
    }
    \caption{PSNR and SSIM for different levels of noise and different reconstructors $\Psi$ for the COULE dataset.}
\end{figure}

\subsection{Example 2: Mayo}\label{ssec:Mayo}

 In this second example, we tested our proposal on an image of the Mayo test set acquired by $n_a = 180$ projections, a detector shape of $n_d = \lfloor 256\sqrt{2} \rfloor$, and  corrupted with  white Gaussian noise with level intensity of~$1\%$. Regarding the parameter selection for the edge-weight function in  \cref{eq:edge-weight-function:applications}, we chose $R=5$ and $\sigma=2\times 10^{-4}$, except for the $\texttt{graphLa+Net}$ method for which we used $R=3$ and $\sigma=10^{-3}$.

Given that the Mayo dataset reflects a real-world scenario, the ground truth images $\gt$, used for generating sinograms and for comparison, are not the actual true images. Instead, they are reconstructions themselves, inherently containing some level of noise. Consequently, comparing metrics for a fixed level of additional noise, as done in \Cref{tab:Coule}, is a bit less informative. Instead, it is still interesting to evaluate the \graphLa method across different reconstructors $\Psi$ and various levels of noise intensity using both PSNR and SSIM metrics, see \Cref{fig:Mayo_stability}. As previously noted in the COULE example, the $\texttt{graphLa+Net}$ method consistently outperforms all other cases, even if the neural network was trained with a $0\%$ noise level.

\begin{figure}\label{fig:Mayo_stability}
    \centering
    \subfigure[PSNR]
    {        
        \includegraphics[width=0.45\textwidth] {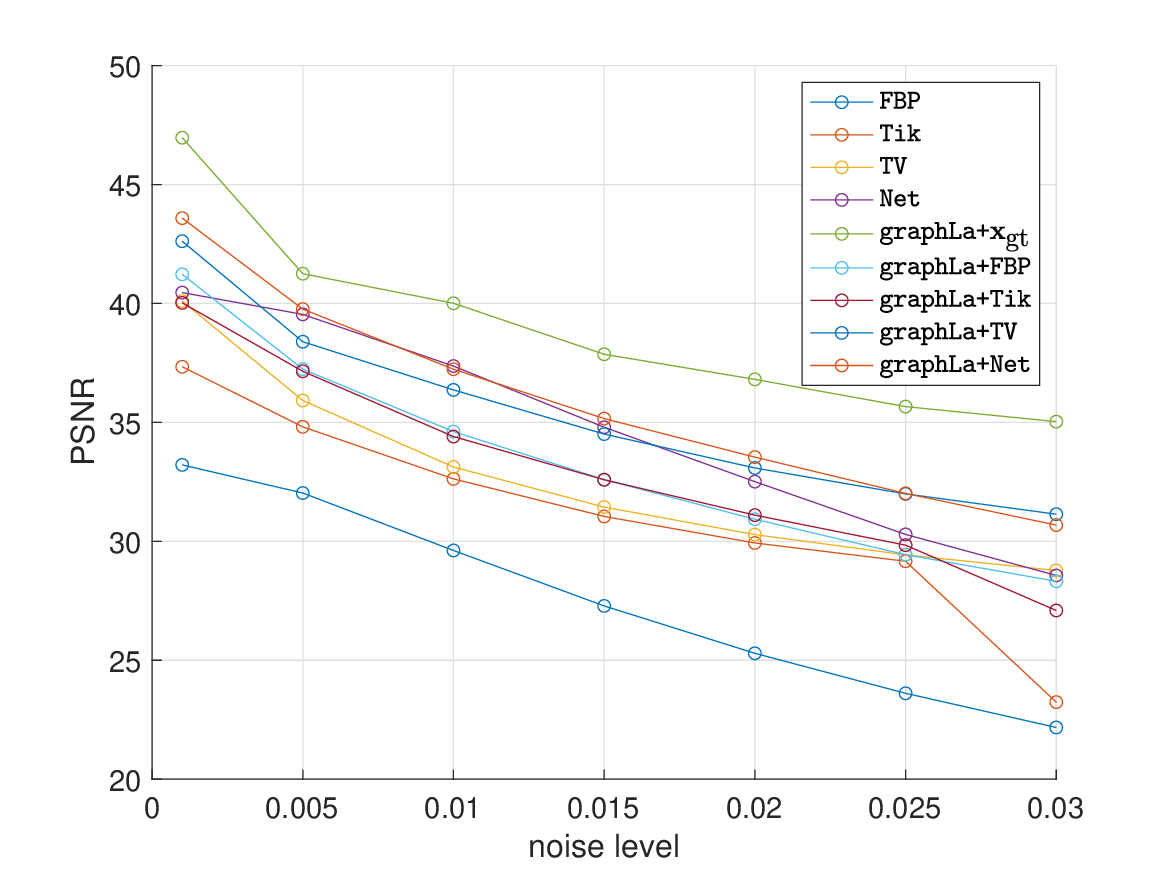} 
    }
    \subfigure[SSIM]
    {        
        \includegraphics[width=0.45\textwidth] {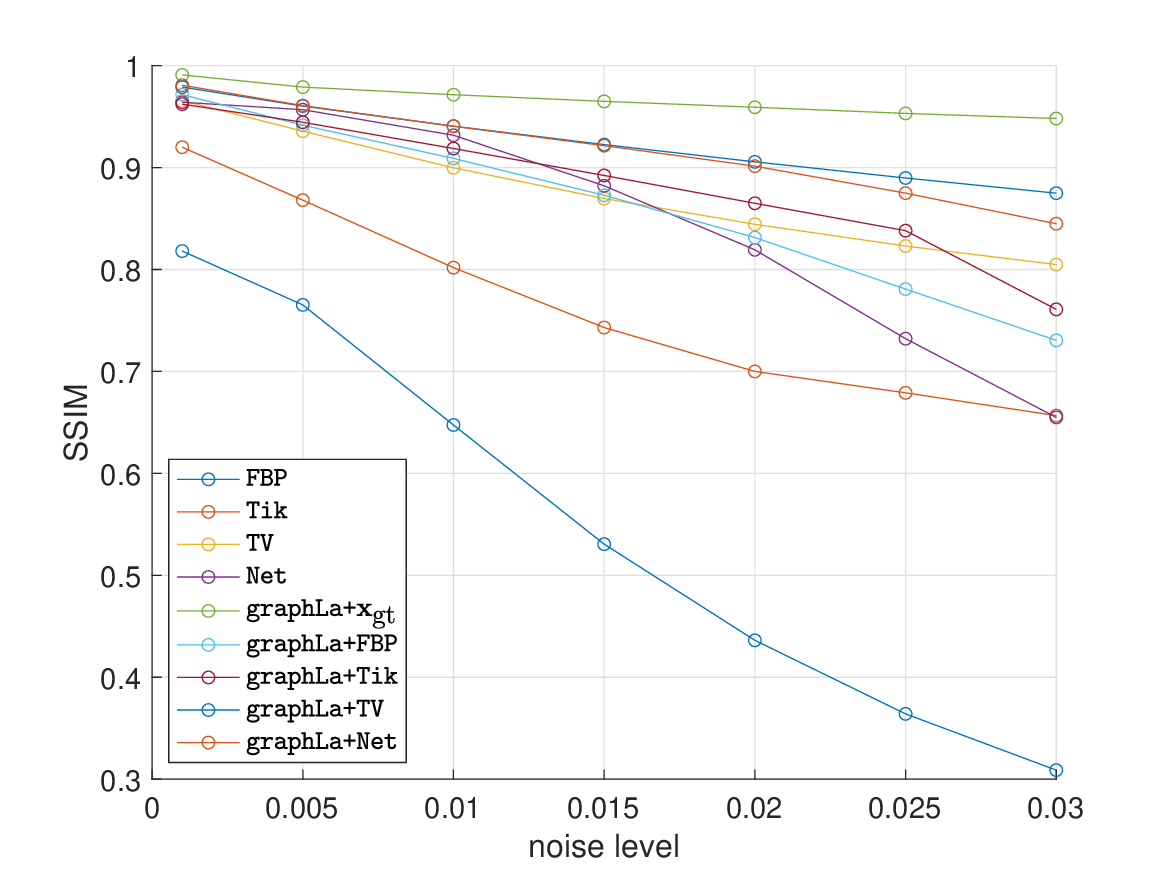} 
    }
    \caption{PSNR and SSIM for different levels of noise and different reconstructors $\Psi$ for the Mayo dataset.}
\end{figure}

 In \Cref{fig:Mayo_rec}, we present a visual inspection of some of the reconstructions. Notably, the \texttt{graphLa+Net} image exhibits the sharpest quality compared to all other cases, and being very close to the upper limit given by $\texttt{graphLa+}\gt$.

\begin{figure}\label{fig:Mayo_rec}
    \centering
    \subfigure[True image $\gt$]
    {
    \begin{tikzpicture}
    \node[anchor=south west,inner sep=0] (image) at (0,0) {\includegraphics[width=0.20\textwidth] {Mayo_true}};
    \begin{scope}[x={(image.south east)},y={(image.north west)}]
        \draw[red,thin] (0.37,0.19) rectangle (0.61,0.47);
    \end{scope}
    \end{tikzpicture}
    }
    \subfigure[\texttt{Tik}]
    {        
        \includegraphics[width=0.20\textwidth] {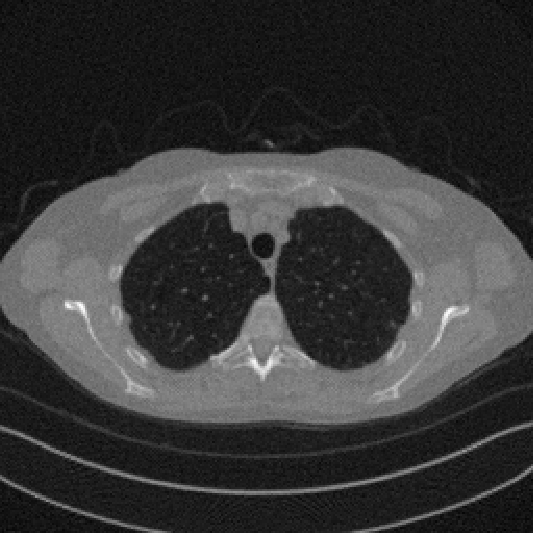} 
    }
    \subfigure[\texttt{TV}]
    {        
        \includegraphics[width=0.20\textwidth] {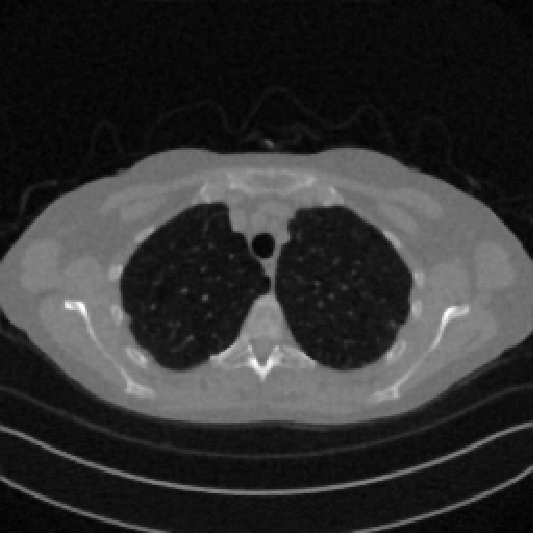} 
    }
    \subfigure[\texttt{Net}]
    {        
        \includegraphics[width=0.20\textwidth] {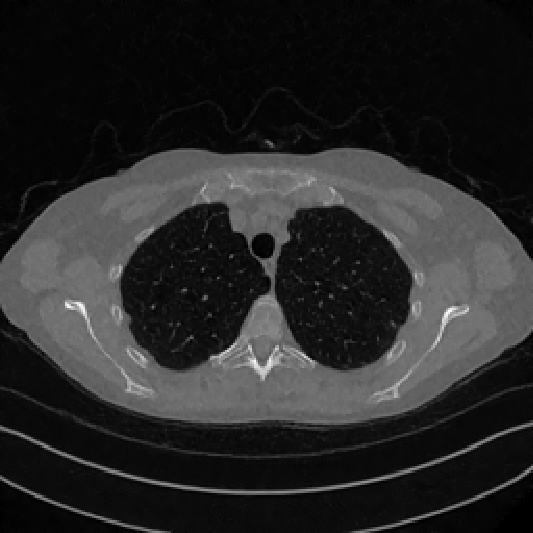} 
    }
    \\
    \subfigure[ $\texttt{graphLa+}\gt$]
    {        
        \includegraphics[width=0.20\textwidth] {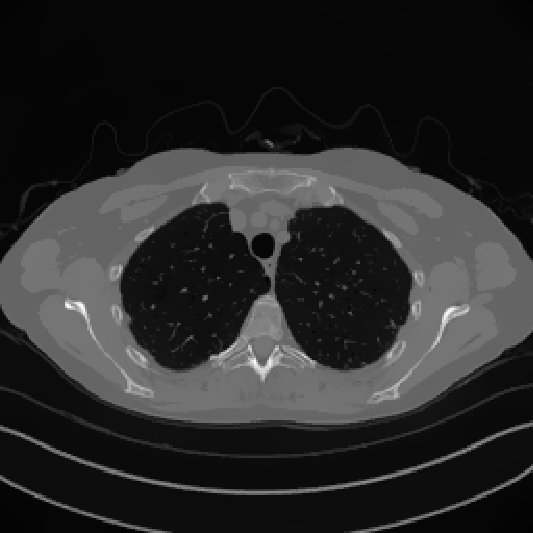} 
    }
    \subfigure[$\texttt{graphLa+Tik}$]
    {        
        \includegraphics[width=0.20\textwidth] {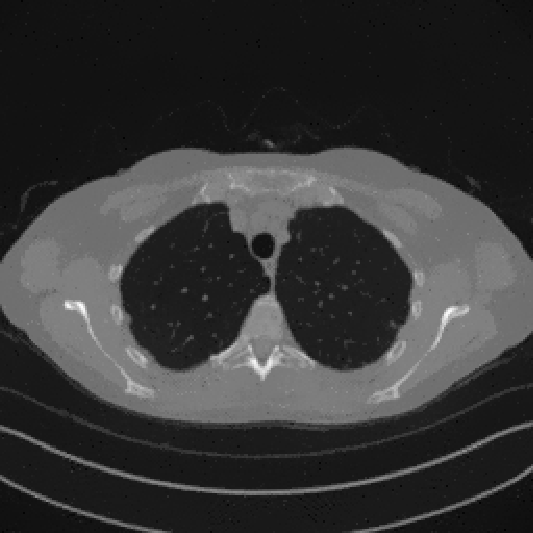} 
    }
    \subfigure[$\texttt{graphLa+TV}$]
    {        
        \includegraphics[width=0.20\textwidth] {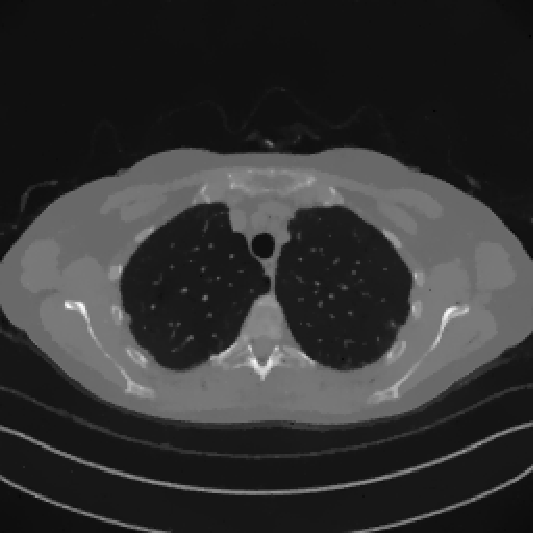} 
    }
    \subfigure[$\texttt{graphLa+Net}$]
    {        
        \includegraphics[width=0.20\textwidth] {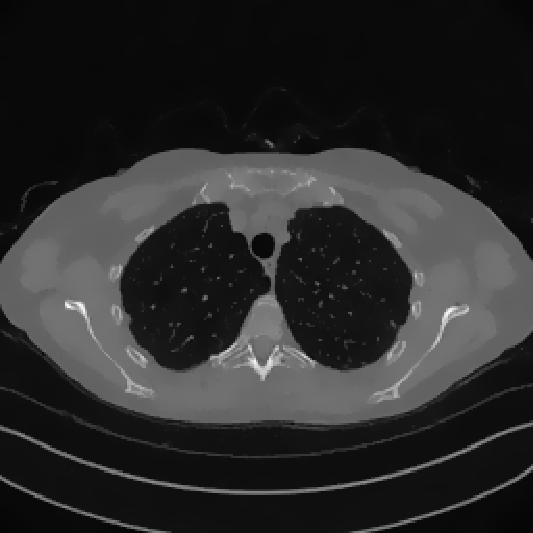} 
    }
    \caption{Initial and final reconstructions using \textnormal{$\texttt{graphLa+}\Psi$} method for different $\Psi$.} 
\end{figure}

As additional confirmation, in Figure \ref{fig:Mayo_zoom} we zoom on the central part of the considered image. In this way, we can clearly note that the $\texttt{graphLa+Net}$ approach achieves also an extraordinary quality of detail in the reconstruction. 

\begin{figure}
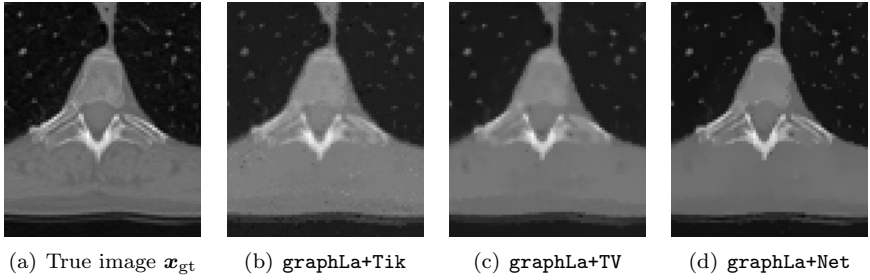
\label{fig:Mayo_zoom}
    \centering
    \subfigure[True image $\gt$]
    {        
        \includegraphics[width=0.20\textwidth,trim={92px 45px 92px 125px},clip] {Mayo_true} 
    }
    \subfigure[$\texttt{graphLa+Tik}$]
    {        
        \includegraphics[width=0.20\textwidth,trim={92px 45px 92px 125px},clip] {Mayo_LaTik} 
    }
    \subfigure[$\texttt{graphLa+TV}$]
    {        
        \includegraphics[width=0.20\textwidth,trim={92px 45px 92px 125px},clip] {Mayo_LaTV} 
    }
    \subfigure[$\texttt{graphLa+Net}$]
    {        
        \includegraphics[width=0.20\textwidth,trim={92px 45px 92px 125px},clip] {Mayo_LaNet} 
    }
    \caption{Zoom in of the central lower part for different methods.} 
\end{figure}

\section{Conclusions}\label{sec:conclusions}
In this work, we introduced and analyzed a novel regularization method that uses a graph Laplacian operator built upon a first approximation of the solution through a reconstructor $\Psi$. We demonstrated that under certain, albeit very weak, hypotheses on the recontructor $\Psi$, \graphLa is a convergent and stable regularization method. In all the proposed numerical examples, \graphLa greatly enhanced the quality of the reconstructions for every initial reconstructor $\Psi$. 

Moreover, leveraging the regularization property of $\texttt{graphLa+}\Psi$, we suggested using a DNN as the first reconstructor $\Psi$. This new hybrid method called $\texttt{graphLa+Net}$, merges the regularity of a standard variational method and the accuracy of a DNN, giving as a result a stable regularization method with very high accuracy.  

In forthcoming works, we will focus more on sharpening the choices for the edge-weight function $w_\bx$ such as building an automatic rule for estimating its parameters.

\appendix
\section{Theoretical analysis}\label{sec:appendix}
%appendix
In this appendix we collect the proofs of the results presented in \Cref{ssec:theoretical_restults}. We will assume that all the hypotheses introduced so far are satisfied, i.e. \Cref{hypothesis:differentiability,hypothesis:existence,hypothesis:null_intersection,hypothesis:null_stability,hypothesis:Psi_convergence,hyp:stability}. 

We begin with the invariance property of the null space of the graph Laplacians built upon $\Psi^\delta_\Theta$.

\vspace{0.1cm}
\noindent\Cref{lem:null_stability}.
\vspace{-0.15cm}
\begin{proof}
The null space of a generic graph Laplacian, as per \Cref{def:graph_Laplacian},  is given by the subspace of functions which are constant on the connected components of the node set $P$. See for example \cite[Lemmas 0.29 and 0.31]{keller2021graphs}.  By \Cref{hypothesis:null_stability}, it is easy to check that a sequence $\{p_i\}_{i=0}^k$ is a walk with respect to $w_{\Psi^\delta_\Theta}$ if and only if it is a walk with respect to $w_0$. Therefore, all the connected components of $P$, identified by $w_{\Psi^\delta_\Theta}$ and $w_0$, are invariant. This concludes the proof.   
\end{proof}

\subsection{Existence of solutions and well-posedness of $\texttt{graphLa+}\Psi$}\label{ssec:existence_well-posedness} 
Define
\begin{equation}\label{def:GammaFunctional}
	\Gamma(\bx) \coloneqq \frac{1}{2}\|K\bx - \by^\delta\|_2^2 + \alpha \|\graphLdelta\bx\|_1.
\end{equation}
We recall that a (nonnegative) functional $\Gamma \colon X \to [0,\infty)$ is \emph{coercive} if $\Gamma(\bx) \to \infty$ for $\|\bx\| \to \infty$, where $\|\cdot\|$  can be any norm on $X\simeq\R^n$. 

Thanks to \Cref{hypothesis:null_intersection}, the functional \cref{def:GammaFunctional} is coercive, for every fixed $\alpha>0$ and $\delta\geq0$.  Although this should certainly be well-known, for the convenience of the reader we include a short proof in the next Lemma~\ref{lem:coercivity}.

\begin{lemma}\label{lem:coercivity}
$\Gamma$ is coercive for every fixed $\alpha>0$ and $\delta\geq 0$.
\end{lemma}
\begin{proof}
	Let $V$ be the invariant null space of $\graphLdelta$ from \Cref{lem:null_stability}, and let us indicate with $\pi$ and $\pi_\perp$ the projection into $V$ and $V^\perp$, respectively. In general, it holds that
	\begin{equation}\label{eq:lem:coercivity:1}
		\qquad     \inf_{\substack{\bu \in V^\perp\\ \bu\neq \boldsymbol{0}}} \frac{\|\graphLdelta\bu\|_1}{\|\bu\|_1} \geq \gamma_1 >0.
	\end{equation}
 Since $\ker(K)\cap V=\{\mathbf{0}\}$ by \Cref{hypothesis:null_intersection}, then it holds that too
	\begin{equation}\label{eq:lem:coercivity:2}
		\qquad     \inf_{\substack{\bv \in V\\ \bv\neq \boldsymbol{0}}} \frac{\|K\bv\|_2}{\|\bv\|_2} \geq \gamma_2 >0.
	\end{equation}
	Fix a sequence $\{\bx_j\}$ such that $\|\bx_j\|_2 \to \infty$. We want to show that $\Gamma(\bx_j)\to \infty$. Clearly, for $j\to \infty$
	\begin{equation*}
		\|\bx_j\|_2 \to \infty \quad \mbox{if and only if} \quad \|\pi\bx_j\|_2^2 + \|\pi_\perp\bx_j\|_1 \to \infty.
	\end{equation*}
	There are two cases:
	\begin{enumerate}[(i)]
		\item\label{lem:coercivity:item1} $\lim_j\|\pi_\perp\bx_j\|_1 =\infty$;
		\item\label{lem:coercivity:item2} $\liminf_j\|\pi_\perp\bx_j\|_1 \leq c < \infty$ and $\lim_j\|\pi\bx_j\|_2^2 =\infty$.
	\end{enumerate}
	If we are in \ref{lem:coercivity:item1}, then  by \cref{eq:lem:coercivity:1}
	\begin{align}
		\alpha\gamma_1\|\pi_\perp\bx_j\|_1 &\leq \alpha\|\graphLdelta\pi_\perp\bx_j\|_1 \nonumber\\
		&\leq   \frac{1}{2}\|K\bx_j - \by^\delta\|_2^2 + \alpha\|\graphLdelta\bx_j\|_1 = \Gamma(\bx_j). \label{eq:lem:coercivity:3}
	\end{align}
	If we are in \ref{lem:coercivity:item2}, then  by \cref{eq:lem:coercivity:2}
	\begin{align}
		\frac{\gamma_2}{4}\|\pi\bx_j\|_2^2 \leq  \frac{1}{4}\|K\pi\bx_j\|_2^2 &\leq \frac{1}{2}\|K\bx_j - \by^\delta\|_2^2 + \alpha\|\graphLdelta\bx_j\|_1 + \frac{1}{2}\|K\pi_\perp\bx_j-\by^\delta\|_2^2 \nonumber\\
		&= \Gamma(\bx_j)  + \frac{1}{2}\|K\pi_\perp\bx_j-\by^\delta\|_2^2. \label{eq:lem:coercivity:4}
	\end{align}
 Notice that from \ref{lem:coercivity:item2} it follows that $\liminf_j\|K\pi_\perp\bx_j-\by^\delta\|_2^2$ is bounded.
 
	Passing to the $\liminf$ in both \cref{eq:lem:coercivity:3,eq:lem:coercivity:4} we conclude. 
\end{proof}	

\begin{remark}\label{rem:unchanged}
The same result as in \Cref{lem:coercivity} applies when replacing $\Psi^\delta_\Theta$ with $\bx_0$ in $\Gamma$, and the proof remains unchanged.
\end{remark}

\vspace{0.1cm}
\noindent\Cref{prop:existence_min}.
\vspace{-0.15cm}
\begin{proof}
	The existence relies on standard topological arguments, but we provide a comprehensive exposition 
	for the convenience of the reader. 
	
	Let 
	$$
	c\coloneqq \inf\{  \|\graphL\bx\|_1 \mid \bx \in X, \; K\bx = \by \},
	$$
	which is well-defined thanks to \Cref{hypothesis:existence}. Therefore there exists a sequence $\{\bx_j\}$ such that $K\bx_j = \by$ for every $j$ and $\lim_j \|\graphL\bx_j\|_1=c$. In particular, there exists $c_1>0$ such that 
	\begin{equation}\label{eq:lem:existence_min_1}
		\|\graphL\bx_j\|_1\leq c_1 \quad \mbox{for every } j.
	\end{equation}
	
	There are two possible cases:
	\begin{enumerate}[(i)]
		\item\label{item:lem:existence_min_1} $\|\bx_j\|_2 \leq c_2$ for some $c_2>0$, for every $j$;
		\item\label{item:lem:existence_min_2} There exists a subsequence $\{\bx_{j'}\}$, such that $\lim_{j'}\|\bx_{j'}\|_2 = \infty$.
	\end{enumerate}
	If we are in case \ref{item:lem:existence_min_1}, then by compactness and continuity arguments we can conclude that there exists $\bx_{\textnormal{sol}}$ such that
	\begin{equation*}
		\lim_{j'} \bx_{j'} = \bx_{\textnormal{sol}}, \quad \mbox{and} \quad \begin{cases}
			K\bx_{\textnormal{sol}}=\by,\\
			\|\graphL\bx_{\textnormal{sol}}\|_1 = c,
		\end{cases} 
	\end{equation*}
	that is, $\bx_{\textnormal{sol}}$ is a graph-minimizing solution with respect to $\bx_0$. 
	
	Suppose now we are in case \ref{item:lem:existence_min_2}. By \Cref{lem:coercivity} and \Cref{rem:unchanged}, $\Gamma(\bx_{j'})\to \infty$ for any fixed $\alpha,\delta$. Since $\|K\bx_{j'}-\by^\delta\|^2_2=\|\by -\by^\delta\|_2^2 \leq \delta^2$ for every $j'$, it follows necessarily that $ \lim_{j'}\|\graphL\bx_{j'}\|_1= \infty$. This leads to an absurdity in light of \eqref{eq:lem:existence_min_1}.
\end{proof}

\vspace{0.1cm}
\noindent\Cref{prop:well-posedness}.
\vspace{-0.15cm}
\begin{proof}
	From \Cref{lem:coercivity}, the nonnegative functional $\Gamma$ is coercive on a finite dimensional vector space. 	By standard theory, there exists a minimizer, see for example \cite[Proposition 11.15]{bauschke2017convex}.
\end{proof}

\vspace{0.1cm}
\noindent\Cref{cor:uniqueness}.
\vspace{-0.15cm}
\begin{proof}
	The uniqueness of $\bx_{\textnormal{sol}}$ is straightforward. On the other hand, the uniqueness of $\sol$ is derived from the fact that if  $K$ is injective then the functional $\bx \mapsto \|K\bx - \by^\delta\|_2^2$ is strongly convex. This property leads to the strong (and therefore strict) convexity of $\Gamma$. According to \cite[Corollary 11.9]{bauschke2017convex}, the desired result follows.
\end{proof}

\subsection{Convergence and stability analysis}\label{ssec:convergence_analysis}

In this subsection, we prove \Cref{thm:convergence,thm:stability}. The main difficulty is given by the regularization term $\mathcal{R}(\bx,\by^\delta)$, which depends on the observed data $\by^\delta$. Therefore, all standard techniques can not be applied straightforwardly.

A crucial role will be played by the following two lemmas, which guarantee uniform convergence of $\graphLdelta$, and a special uniform coercivity property for $\Gamma$ in \cref{def:GammaFunctional}.

\begin{lemma}\label{lem:stability}
	Let $\Theta=\Theta(\delta,\by^\delta)$ and  $\bx_0$   be defined as in \Cref{hypothesis:Psi_convergence}. For every $\bx\in X$ it holds that 
$$
\|\graphLdelta\bx - \graphL\bx\|_1 \leq c \| \bx \|_1 \|  \Psi_\Theta^\delta - \bx_0 \|_2\to 0 \quad \mbox{as } \delta\to 0, 
$$
where $c$ is a positive constant independent of $\bx$. 
\end{lemma}
\begin{proof}
Indicating with $a^\delta_{pq}$ the elements of the matrix $A^\delta\coloneqq \graphLdelta- \graphL$, then
\begin{equation}\label{eq:lem:stability6}
	\|\graphLdelta\bx - \graphL\bx\|_1 \leq \|A^\delta\|\| \bx \|_1 =   \left(\max_{q\in P}\sum_{p\in P} |a^\delta_{pq}|\right) \|\bx \|_1 ,
\end{equation}
where $\|A^\delta\|$ is the induced matrix $1$-norm. Making explicit now the values of $a^\delta_{pq}$, we have that
\begin{align}
	\sum_{p\in P} |a^\delta_{pq}| &=  |a^\delta_{qq}| + \sum_{\substack{p\in P\\p\neq q}} |a^\delta_{pq}|\nonumber\\
	&= \left| \sum_{\substack{\ell\in P\\\ell\neq q}} \frac{w_{\Psi_\Theta^\delta}(q,\ell)}{\mu_{\Psi_\Theta^\delta}(q)} - \frac{w_0(q,\ell)}{\mu_0(q)}\right| + \sum_{\substack{p\in P\\p\neq q}}  \left| \frac{w_{\Psi_\Theta^\delta}(p,q)}{\mu_{\Psi_\Theta^\delta}(p)} - \frac{w_0(p,q)}{\mu_0(p)}\right|\nonumber\\
	&\leq \sum_{\substack{p\in P\\p\neq q}} \left| \frac{w_{\Psi_\Theta^\delta}(p,q)}{\mu_{\Psi_\Theta^\delta}(q)} - \frac{w_0(p,q)}{\mu_0(q)}\right| + \sum_{\substack{p\in P\\p\neq q}}  \left| \frac{w_{\Psi_\Theta^\delta}(p,q)}{\mu_{\Psi_\Theta^\delta}(p)} - \frac{w_0(p,q)}{\mu_0(p)}\right| \label{eq:lem:stability0} 
\end{align}
where in the last inequality we used the symmetry of the edge-weight functions $w_{\Psi_\Theta^\delta}$ and $w_0$. To simplify the notation, define 
$$
t_{\delta,p,q}\coloneqq |\Psi^\delta_\Theta(p) - \Psi^\delta_\Theta(q)|, \quad t_{0,p,q}\coloneqq |\bx_0(p) - \bx_0(q)|.
$$
Let us observe that, for every fixed triple $p, q, k \in P$,
\begin{align}
	  \left| \frac{w_{\Psi_\Theta^\delta}(p,q)}{\mu_{\Psi_\Theta^\delta}(k)} - \frac{w_0(p,q)}{\mu_0(k)}\right| &= w_{\textnormal{d}}(p,q)\left|\frac{h_{\textnormal{i}}(t_{\delta,p,q})}{\mu_{\Psi_\Theta^\delta}(k)} - \frac{h_{\textnormal{i}}(t_{0,p,q})}{\mu_0(k)}\right|\nonumber \\
	  	&= \frac{w_{\textnormal{d}}(p,q)}{\mu_{\Psi_\Theta^\delta}(k)\mu_0(k)}\left|\mu_0(k)h_{\textnormal{i}}(t_{\delta,p,q}) - \mu_{\Psi_\Theta^\delta}(k)h_{\textnormal{i}}(t_{0,p,q})\right|.\label{eq:lem:stability1}
\end{align}
Let us recall now that, by  \Cref{hypothesis:differentiability}, we have that
\begin{equation}\label{eq:lem:stability2}
|h_{\textnormal{i}}(t_{\delta,p,q}) - h_{\textnormal{i}}(t_{0,p,q})| \leq L' |t_{\delta,p,q} - t_{0,p,q}| , \quad |\mu_{\Psi_\Theta^\delta}(k) - \mu_0(k)| \leq L''_k\|\Psi_\Theta^\delta - \bx_0 \|_2.
\end{equation}
By adding and subtracting the auxiliary term $\mu_0(k)h_{\textnormal{i}}(t_{0,p,q})$,
it holds that
\begin{flalign}
	&\resizebox{1\textwidth}{!}{%
		$
		\begin{aligned}
			\left|\mu_0(k)h_{\textnormal{i}}(t_{\delta,p,q}) - \mu_{\Psi_\Theta^\delta}(k)h_{\textnormal{i}}(t_{0,p,q})\right| 
			&\leq \mu_0(k)|h_{\textnormal{i}}(t_{\delta,p,q}) - h_{\textnormal{i}}(t_{0,p,q})| + |\mu_{\Psi_\Theta^\delta}(k) - \mu_0(k)|h_{\textnormal{i}}(t_{0,p,q})\\
			&\leq L'\max_{k}\{\mu_0(k)\}  |t_{\delta,p,q} - t_{0,p,q}| + \max_k \{L''_k\}\max_{p,q}\{h_{\textnormal{i}}(t_{0,p,q})\}\|\Psi_\Theta^\delta - \bx_0 \|_2.
		\end{aligned}
		$}
	& \label{eq:lem:stability:first_bound}
\end{flalign}
Let us bound now  $|t_{\delta,p,q} - t_{0,p,q}|$:
\begin{align}
	|t_{\delta,p,q} - t_{0,p,q}| &= \left| |\Psi^\delta_\Theta(p) -\Psi^\delta_\Theta(q) | - |\bx_0(p) -\bx_0(q) |  \right|\nonumber\\
	&\leq \left| \left(\Psi^\delta_\Theta(p)-\bx_0(p)\right)  +  \left(\bx_0(q)-\Psi^\delta_\Theta(q)\right)\right|\nonumber\\
	&\leq \left|\Psi^\delta_\Theta(p)-\bx_0(p)  \right| + \left|\Psi^\delta_\Theta(q)-\bx_0(q)  \right|\nonumber\\
	&\leq 2\|\Psi^\delta_\Theta -\bx_0\|_2.\label{eq:lem:stability:second_bound}
\end{align}
Therefore, indicating with
\begin{eqnarray*}
\underline{\mu}_\delta \coloneqq \min_k\{\mu_{\Psi_\Theta^\delta}(k)\}, \quad \underline{\mu}_0\coloneqq  \min_k\{\mu_0(k)\}, \\
\overline{\mu}_0 \coloneqq \max_k\{\mu_0(k)\}, \quad
\overline{L''}\coloneqq  \max_k \{L''_k\}, \quad \overline{h}_{\textnormal{i}}\coloneqq  \max_{p,q}\{h_{\textnormal{i}}(t_{0,p,q})\},
\end{eqnarray*}
and using \Cref{eq:lem:stability:first_bound,eq:lem:stability:second_bound} into \Cref{eq:lem:stability1}, 
\begin{equation}
	  \left| \frac{w_{\Psi_\Theta^\delta}(p,q)}{\mu_{\Psi_\Theta^\delta}(k)} - \frac{w_0(p,q)g(\delta, \bx_0)}{\mu_0(k)}\right|  \leq w_{\textnormal{d}}(p,q)\frac{2L'\overline{\mu}_0 + \overline{L''}\,\overline{h}_{\textnormal{i}}}{\underline{\mu}_\delta\,\underline{\mu}_0}\|\Psi^\delta_\Theta -\bx_0\|_2. \label{eq:lem:stability4} 
\end{equation}
Now, define
\begin{equation*}
	\overline{d} \coloneqq \max_{p,q}\{w_{\textnormal{d}}(p,q)\}, \qquad \overline{\kappa} \coloneqq \max_q\left\{\operatorname{card}\{ p\in P \mid w_{\textnormal{d}}(p,q) \neq 0  \}\right\}.
\end{equation*}
Then, combining \Cref{eq:lem:stability0,eq:lem:stability4}, for every fixed $q$ it holds that
\begin{align}\label{eq:lem:stability5}
\sum_{p\in P} |a^\delta_{pq}|  &\leq  \frac{2\left( 2L'\overline{\mu}_0 + \overline{L''}\,\overline{h}_{\textnormal{i}}\right)}{\underline{\mu}_\delta\,\underline{\mu}_0}\|\Psi^\delta_\Theta -\bx_0\|_2\sum_{p\in P}w_{\textnormal{d}}(p,q)\nonumber\\
&\leq  \frac{2\overline{d}  \overline{\kappa} \left( 2L'\overline{\mu}_0 + \overline{L''}\,\overline{h}_{\textnormal{i}}\right)}{\underline{\mu}_\delta\,\underline{\mu}_0}\|\Psi^\delta_\Theta -\bx_0\|_2.
\end{align}
Finally, from  \Cref{eq:lem:stability2,hypothesis:Psi_convergence},  there exists $\delta_0$ such that
\begin{equation*}
\underline{\mu}_\delta  \geq \frac{1}{2} \underline{\mu}_0 \quad \mbox{for every}\quad 0\leq \delta\leq \delta_0,
\end{equation*}
and by defining
\begin{equation*}
	c \coloneqq \frac{4\overline{d} \overline{\kappa}\left( 2L'\overline{\mu}_0 + \overline{L''}\,\overline{h}_{\textnormal{i}}\right)}{\underline{\mu}_0^2},
\end{equation*}
the thesis follows from \Cref{eq:lem:stability6,eq:lem:stability5}.
\end{proof}

\begin{remark}\label{rem:constant_c}
The constant $c$ may depend on the dimension~$n$. However, by selecting appropriate values for $w_{\textnormal{d}}$, $h_{\textnormal{i}}$ and $\mu_\bx$, which vary based on specific applications,  it is possible to make $c$ independent of the dimension $n$ of the vector space $C(P) \simeq X$. For example, if we fix the edge-weight function  $w_{\textnormal{d}}$ independently of $n$ and set $\mu_\bx \equiv \mu$ with $\mu>0$ as a positive constant for every $\bx$, then $\overline{d}\overline{\kappa}$ becomes independent of $n$ and $\overline{L''}=0$, thereby making $c$  independent of $n$ as well.

	The choices made for the numerical experiments in this manuscript ensure that $c$ is independent of $n$. For more details, refer to \Cref{ssec:hp_validity}.
\end{remark}

\begin{lemma}\label{lem:uniform_coercivity}
	Let $\Theta=\Theta(\delta,\by^\delta)$ be the parameter choice rule as in \Cref{hypothesis:Psi_convergence} and let $\delta_k \to 0$. Write $\Theta_k \coloneqq \Theta(\delta_k, \by^{\delta_k})$ and fix $\alpha>0$. For any sequence $\{\bx_k\}$ such that $\limsup_k\|\bx_k\|_2 =\infty$, then
	\begin{equation*}
		\limsup_k \frac{1}{2}\|K\bx_k - \by\|_2^2 + \alpha \|\graphLdeltak \bx_k\|_1 = \infty.
	\end{equation*}
\end{lemma}	
\begin{proof}
	Let $V$ be the invariant null space from \Cref{lem:null_stability}. Define
	\begin{equation*}
		\gamma_k \coloneqq  \inf_{\substack{\bu \in V^\perp\\\|\bu\|_1=1}}\|\graphLdeltak \bu\|_1>0 \quad \mbox{and} \quad \gamma_0 \coloneqq \inf_{\substack{\bu \in V^\perp\\\|\bu\|_1=1}}\|\graphL \bu\|_1 >0.
	\end{equation*}
	By \Cref{lem:stability}, it holds that
	\begin{equation*}
		\forall\, \bu\in V^\perp \mbox{ s.t. } \|\bu\|_1=1, \quad \|\graphL \bu\|_1 - \frac{\gamma_0}{2}   \leq  \|\graphLdeltak \bu\|_1  \qquad \mbox{for } k\geq N=N(\gamma_0).
	\end{equation*}
	Therefore, 
	\begin{equation*}
		\gamma_k \geq \frac{\gamma_0}{2} \qquad \forall k \geq N(\gamma_0).
	\end{equation*}
	In particular, there exists $\hat{\gamma}>0$ such that
	\begin{equation*}
		\|\graphLdeltak \bx\|_1 \geq\hat{\gamma}\|\pi_\perp\bx\|_1 \qquad \forall\, \bx,\; \forall\, k.
	\end{equation*}
	The rest of the proof follows like in \Cref{lem:coercivity}.
\end{proof}		

The next theorem presents a convergence result for the \graphLa method  \cref{graphModel}. The overall proof follows a fairly standard approach, as for example in \cite[Theorem 3.26]{scherzer2009variational}. However, since the regularizing term in \cref{graphModel} depends on the data $\by^\delta$ as well, it involves a few nontrivial technical aspects. 

We will use a slight modification of the notation introduced in \cref{def:GammaFunctional}. Specifically,
\begin{equation*}
	\Gamma_k(\bx)\coloneqq \frac{1}{2}\|K\bx - \by^{\delta_k}\|_2^2 + \alpha_k \|\graphLdeltak\bx\|_1.
\end{equation*}

\vspace{0.1cm}
\noindent\Cref{thm:convergence}.
\vspace{-0.15cm}
\begin{proof}
	The sequence $\{\bx_k\}$ is well-posed thanks to \Cref{prop:well-posedness}. Fix a graph-minimizing solution $\bx_{\textnormal{sol}}$ as in \Cref{def:graph-minimizing_solution}, which exists because of \Cref{prop:existence_min}.  Then, by definition, it holds that
	\begin{align}\label{eq:thm:convergence:c}
		\Gamma_k(\bx_k) &\leq \Gamma_k(\bx_{\textnormal{sol}}) \leq \frac{\delta_k^2}{2} + \alpha_k \|\graphL\bx_{\textnormal{sol}}\|_1 + \alpha_kc\|\bx_{\textnormal{sol}}\|_1 \|\Psi^{\delta_k}_{\Theta_k} - \bx_0 \|_2 \to 0
	\end{align}
as $k \to \infty$,	where the last inequality comes from  \Cref{lem:stability}, and the convergence to zero is granted by \cref{eq:thm:convergence:a}. Therefore, $\|K\bx_k - \by^{\delta_k}\|_2^2 \to 0$, and in particular 
	\begin{equation}\label{eq:thm:convergence:d}
		\|K\bx_k - \by\|_2 \leq \|K\bx_k - \by^{\delta_k}\|_2 + \delta_k \to 0 \quad \mbox{as } k\to \infty.
	\end{equation}
	Since 
	$$
	\alpha_k\|\graphLdeltak\bx_k\|_1 \leq \Gamma_k(\bx_k),
	$$
	then by \cref{eq:thm:convergence:b,eq:thm:convergence:c},  we get that
	\begin{equation}\label{eq:thm:convergence:e}
		\limsup_k\|\graphLdeltak\bx_k\|_1 \leq \|\graphL\bx_{\textnormal{sol}}\|_1.
	\end{equation}
	Let $\alpha^+ \coloneqq \max\{\alpha_k \mid k \in \N \}$. Then, combining \cref{eq:thm:convergence:d,eq:thm:convergence:e},
	\begin{equation*}
		\frac{1}{2}\|K\bx_k - \by\|_2^2 + \alpha^+\|\graphLdeltak\bx_k\|_1 \leq  c < \infty,
	\end{equation*}
	and from \Cref{lem:uniform_coercivity} we deduce that $\{\bx_k\}$ is bounded. Therefore, there exists a convergent subsequence $\{\bx_{k'}\}$ which converges to a point $\bx^*$. From \cref{eq:thm:convergence:d}, it holds that $K\bx^* = \by$. Moreover, by the boundedness of $\{\bx_{k'}\}$ and the uniform convergence granted by \Cref{lem:stability}, we infer that
	\begin{equation*}
		\|\graphL\bx^*\|_1 = \lim_{k'}\|\Delta_{\Psi^{\delta_{k'}}_{\Theta_{k'}}}\bx_{k'}\|_1.
	\end{equation*}
	Applying \cref{eq:thm:convergence:e}, we finally get
	\begin{equation*}
		\|\graphL\bx^*\|_1 = \lim_{k'}\|\Delta_{\Psi^{\delta_{k'}}_{\Theta_{k'}}}\bx_{k'}\|_1	\leq \|\graphL\bx_{\textnormal{sol}}\|_1 \leq \|\graphL\bx^*\|_1.
	\end{equation*}
	That is, $\bx^*$ is a graph-minimizing solution. If the graph-minimization solution is unique, then we have just proven that every subsequence of $\{\bx_k\}$ has a subsequence converging to $\bx_{\textnormal{sol}}$, and therefore $\bx_k \to \bx_{\textnormal{sol}}$ by a standard topological argument. 
\end{proof}

We proceed now to prove the stability result in \Cref{thm:stability}. However, as the first step we demonstrate that the standard Tikhonov reconstruction method coupled with the discrepancy principle satisfies \Cref{hyp:stability}.
\begin{example}\label{example:Tikhonov}
    Consider a standard Tikhonov reconstruction method, that is
\begin{equation*}
	\Psi_{\Theta}(\by) \coloneqq \left(K^TK +\Theta I\right)^{-1}K^T\by,
\end{equation*}
where $\Theta \in (0,\infty)$. Let $\Theta = \Theta(\delta, \by^\delta)$ defined by the discrepancy principle as per \cite[Equation (4.57)]{engl1996regularization}. Then   $\Theta_k \to \Theta=\Theta(\delta,\by^{\delta}) \ \mathrm{for} \ k\rightarrow \infty$. Setting
\begin{equation*}
    \mathcal{T}^k = (K^TK+\Theta_kI)^{-1} \quad \mbox{and} \quad   \mathcal{T}=(K^TK+\Theta I)^{-1},
\end{equation*}
it holds that
\begin{flalign*}
	\|\Psi_{\Theta_k}(\by^{\delta_k})-\Psi_{\Theta}(\by^{\delta})\|_2\le \underbrace{\|(\mathcal{T}^k-\mathcal{T})K^T\by^{\delta_k}\|_2}_{\mathbf{I}} + \underbrace{\|\mathcal{T}K^T(\by^{\delta}-\by^{\delta_k})\|_2}_{\mathbf{II}}. 
\end{flalign*}
Now,
\begin{flalign*}
	\mathbf{I}\le |\Theta -\Theta_k|\|\mathcal{T}^k\| \|\mathcal{T}K^T\|\|\by^{\delta_k}\|_2\le \frac{|\Theta-\Theta_k|}{2\Theta_k\sqrt{\Theta}}\|\by^{\delta_k}\|_2 \to 0,
\end{flalign*} 
and
\begin{flalign*}
	\mathbf{II}\le\|(K^TK+\Theta I)^{-1}K^T\|\|\by^{\delta}-\by^{\delta_k}\|_2\le\frac{1}{2\sqrt{\Theta}}\|\by^{\delta}-\by^{\delta_k}\| \to 0.
\end{flalign*}
\end{example}
For the proof of \Cref{thm:stability} we need a couple of preliminary lemmas.
\begin{lemma}\label{lemma:Gamma_inequalities}
	For all $\bx\in X$ and $\by^{\delta_{k_1}},\by^{\delta_{k_2}}\in Y$, we have
	\begin{equation*}
		\Gamma_{k_1}(\bx)\le 2\Gamma_{k_2}(\bx)+\|\by^{\delta_{k_1}}-\by^{\delta_{k_2}}\|_2^2+\|(\Delta_{\Psi_{\Theta_{k_1}}^{\delta_{k_1}}}-\Delta_{\Psi_{\Theta_{k_2}}^{\delta_{k_2}}})\bx\|_1
	\end{equation*}
\end{lemma}
\begin{proof}
	By standard $p$-norm inequalities 
	\begin{flalign*}
		\Gamma_{k_1}(\bx)&=\frac{1}{2}\|K\bx-\by^{\delta_{k_1}}\|_2^2+\alpha\|\Delta_{\Psi_{\Theta_{k_1}}^{\delta_{k_1}}}\bx\|_1\\& \le \|K\bx-\by^{\delta_{k_2}}\|_2^2+\|\by^{\delta_{k_1}}-\by^{\delta_{k_2}}\|_2^2+\alpha\|\Delta_{\Psi_{\Theta_{k_1}}^{\delta_{k_1}}}\bx\|_1\\&\le2\Gamma_{k_2}(\bx)+\|\by^{\delta_{k_1}}-\by^{\delta_{k_2}}\|_2^2+\|(\Delta_{\Psi_{\Theta_{k_1}}^{\delta_{k_1}}}-\Delta_{\Psi_{\Theta_{k_2}}^{\delta_{k_2}}})\bx\|_1.
	\end{flalign*}  
\end{proof}

\begin{lemma}\label{lemma:Lapl_bound}
	Let ${\delta_k}$, $\Theta_k$ and $\Theta$ defined as in  \Cref{hyp:stability}. It holds that 
	$$
	\|\Delta_{\Psi_{\Theta_k}^{\delta_k}}\bx-\Delta_{\Psi_{\Theta}^{\delta}}\bx\|_1 \leq c\| \bx \|_1 \|\Psi^{\delta_k}_{\Theta_k} -\Psi_{\Theta}^{\delta}\|_2 \to 0 \quad \mbox{as } k\to \infty,
	$$
 where $c$ is a positive constant independent of $\bx$.
\end{lemma}
\begin{proof}
	The proof is similar to \Cref{lem:stability} using  \Cref{hyp:stability}.
\end{proof}

\vspace{0.1cm}
\noindent\Cref{thm:stability}.
\vspace{-0.15cm}
\begin{proof}
	Because $\bx_k$ is a minimizer of $\Gamma_k$, we have 
	\begin{equation}
		\Gamma_k(\bx_k)\le \Gamma_k(\bx),\quad \forall\,\bx\in X.
	\end{equation}
	Choose now a vector $\bar{\bx}\in X$. By applying the previous equation to $\bx=\bar{\bx}$ and using twice  \Cref{lemma:Gamma_inequalities}, it follows that
	\begin{flalign*}
		\Gamma(\bx_k)&\le 2\Gamma_k(\bx_k)+\|\by^{\delta}-\by^{\delta_k}\|_2^2+\|(\Delta_{\Psi_{\Theta}^{\delta}}-\Delta_{\Psi_{\Theta_k}^{\delta_k}})\bx_k\|_1\\&\le2\Gamma_k(\bar{\bx})+\|\by^{\delta}-\by^{\delta_k}\|_2^2+\|(\Delta_{\Psi_{\Theta}^{\delta}}-\Delta_{\Psi_{\Theta_k}^{\delta_k}})\bx_k\|_1\\&\le 4\Gamma(\bar{\bx})+3\|\by^{\delta}-\by^{\delta_k}\|_2^2+2\|(\Delta_{\Psi_{\Theta}^{\delta}}-\Delta_{\Psi_{\Theta_k}^{\delta_k}})\bar{\bx}\|_1+\|(\Delta_{\Psi_{\Theta}^{\delta}}-\Delta_{\Psi_{\Theta_k}^{\delta_k}})\bx_k\|_1\\
  &\leq 4\Gamma(\bar{\bx})+3\|\by^{\delta}-\by^{\delta_k}\|_2^2 + (2\|\bar{\bx}\|_1 + \|\bx_k\|_1)\|\Psi_\Theta^\delta - \Psi_{\Theta_k}^{\delta_k}\|_2,
	\end{flalign*}
where we used \Cref{lemma:Lapl_bound} in the last inequality. Arguing as in the proof of \Cref{thm:convergence} and adapting \Cref{lem:uniform_coercivity}, it is possible to show that $\{\bx_k\}$ is bounded.	From  \Cref{hyp:stability}, and since $\by^{\delta_k}$ converges to $\by^{\delta}$,  then there exist $k_0\in\mathbb{N}$ such that
	\begin{equation*}
		M:=4\Gamma(\bar{\bx})+1\ge \Gamma(\bx_k),\quad \forall\, k\ge k_0.
	\end{equation*}
 The thesis now follows  by adapting the proof in  \cite[Theorem 3.23]{scherzer2009variational}.
\end{proof}

\section*{Acknowledgments}
We thank the anonymous referees for their valuable feedback, which improved this manuscript.

\bibliographystyle{siamplain}
\bibliography{graphLa+.bib}	

\begin{thebibliography}{10}

\bibitem{alberti2021learning}
{\sc G.~S. Alberti, E.~De~Vito, M.~Lassas, L.~Ratti, and M.~Santacesaria}, {\em
  Learning the optimal tikhonov regularizer for inverse problems}, Advances in
  Neural Information Processing Systems, 34 (2021), pp.~25205--25216.

\bibitem{devangelista2023graphlaplus}
{\sc S.~Aleotti, D.~Bianchi, D.~Evangelista, M.~Donatelli, W.~Li, and
  E.~Loli~Piccolomini}, {\em Official {GitHub} repository for the
  $\textnormal{\texttt{graphla+}}\psi$ codes}, 2023,
  \url{https://github.com/devangelista2/GraphLaPlus}.
\newblock Online; accessed 26-July-2024.

\bibitem{aleotti2024fractional}
{\sc S.~Aleotti, A.~Buccini, and M.~Donatelli}, {\em {Fractional graph
  Laplacian for image reconstruction}}, Applied Numerical Mathematics, 200
  (2024), pp.~43--57.

\bibitem{ali2019generalized}
{\sc A.~Ali and R.~J. Tibshirani}, {\em {The Generalized Lasso Problem and
  Uniqueness}}, Electronic Journal of Statistics, 13 (2019), pp.~2307--2347.

\bibitem{antun2020instabilities}
{\sc V.~Antun, F.~Renna, C.~Poon, B.~Adcock, and A.~C. Hansen}, {\em On
  instabilities of deep learning in image reconstruction and the potential
  costs of {AI}}, in Proceedings of the National Academy of Sciences, vol.~117,
  2020, pp.~30088--30095.

\bibitem{arias2009variational}
{\sc P.~Arias, V.~Caselles, and G.~Sapiro}, {\em {A variational framework for
  non-local image inpainting}}, in International Workshop on Energy
  Minimization Methods in Computer Vision and Pattern Recognition, Springer,
  2009, pp.~345--358.

\bibitem{arridge2019solving}
{\sc S.~Arridge, P.~Maass, O.~{\"O}ktem, and C.-B. Sch{\"o}nlieb}, {\em
  {Solving inverse problems using data-driven models}}, Acta Numerica, 28
  (2019), pp.~1--174.

\bibitem{bauschke2017convex}
{\sc H.~H. Bauschke and P.~L. Combettes}, {\em {Convex Analysis and Monotone
  Operator Theory in Hilbert Spaces}}, CMS books in mathematics, Springer Cham,
  2~ed., 2017.

\bibitem{bianchi2021graph}
{\sc D.~Bianchi, A.~Buccini, M.~Donatelli, and E.~Randazzo}, {\em {Graph
  Laplacian for image deblurring}}, ETNA, 55 (2022), pp.~169--186.

\bibitem{bianchi2015iterated}
{\sc D.~Bianchi, A.~Buccini, M.~Donatelli, and S.~Serra-Capizzano}, {\em
  {Iterated fractional Tikhonov regularization}}, Inverse Problems, 31 (2015),
  p.~055005.

\bibitem{bianchi2017generalized}
{\sc D.~Bianchi and M.~Donatelli}, {\em {On generalized iterated Tikhonov
  regularization with operator-dependent seminorms}}, Electronic Transactions
  on Numerical Analysis, 47 (2017), pp.~73--99.

\bibitem{bianchi2021graph_approximation}
{\sc D.~Bianchi and M.~Donatelli}, {\em {Graph approximation and generalized
  Tikhonov regularization for signal deblurring}}, in 2021 21st International
  Conference on Computational Science and Its Applications (ICCSA), IEEE, 2021,
  pp.~93--100.

\bibitem{bianchi2023graph}
{\sc D.~Bianchi, M.~Donatelli, D.~Evangelista, W.~Li, and E.~L. Piccolomini},
  {\em {Graph Laplacian and Neural Networks for Inverse Problems in Imaging:
  GraphLaNet}}, in International Conference on Scale Space and Variational
  Methods in Computer Vision, 2023, pp.~175--186.

\bibitem{bianchi2023uniformly}
{\sc D.~Bianchi, G.~Lai, and W.~Li}, {\em {Uniformly convex neural networks and
  non-stationary iterated network Tikhonov (iNETT) method}}, Inverse Problems,
  39 (2023), p.~055002.

\bibitem{bishop2024deep}
{\sc C.~M. Bishop and H.~Bishop}, {\em {Deep learning: Foundations and
  concepts}}, vol.~1, Springer, 2024.

\bibitem{bronstein2017geometric}
{\sc M.~M. Bronstein, J.~Bruna, Y.~LeCun, A.~Szlam, and P.~Vandergheynst}, {\em
  {Geometric Deep Learning: Going beyond Euclidean data}}, IEEE Signal
  Processing Magazine, 34 (2017), pp.~18--42.

\bibitem{buccini2021graph}
{\sc A.~Buccini and M.~Donatelli}, {\em {Graph Laplacian in $\ell^2-\ell^q$
  regularization for image reconstruction}}, in 2021 21st International
  Conference on Computational Science and Its Applications (ICCSA), IEEE, 2021,
  pp.~29--38.

\bibitem{buccini2023limited}
{\sc A.~Buccini and L.~Reichel}, {\em Limited memory restarted $\ell^p-\ell^q$
  minimization methods using generalized {K}rylov subspaces}, Advances in
  Computational Mathematics, 49 (2023), p.~26.

\bibitem{cai2008framelet}
{\sc J.-F. Cai, R.~H. Chan, and Z.~Shen}, {\em A framelet-based image
  inpainting algorithm}, Applied and Computational Harmonic Analysis, 24
  (2008), pp.~131--149.

\bibitem{cai2016regularization}
{\sc Y.~Cai, M.~Donatelli, D.~Bianchi, and T.-Z. Huang}, {\em Regularization
  preconditioners for frame-based image deblurring with reduced boundary
  artifacts}, SIAM Journal on Scientific Computing, 38 (2016), pp.~B164--B189.

\bibitem{calatroni2017graph}
{\sc L.~Calatroni, Y.~van Gennip, C.-B. Sch{\"o}nlieb, H.~M. Rowland, and
  A.~Flenner}, {\em {Graph clustering, variational image segmentation methods
  and Hough transform scale detection for object measurement in images}},
  Journal of Mathematical Imaging and Vision, 57 (2017), pp.~269--291.

\bibitem{candes2013simple}
{\sc E.~Candes and B.~Recht}, {\em Simple bounds for recovering low-complexity
  models}, Mathematical Programming, 141 (2013), pp.~577--589.

\bibitem{chung2019flexible}
{\sc J.~Chung and S.~Gazzola}, {\em Flexible {K}rylov methods for $\ell_p$
  regularization}, SIAM Journal on Scientific Computing, 41 (2019),
  pp.~S149--S171.

\bibitem{colbrook2022difficulty}
{\sc M.~J. Colbrook, V.~Antun, and A.~C. Hansen}, {\em {The difficulty of
  computing stable and accurate neural networks: On the barriers of deep
  learning and Smale’s 18th problem}}, in Proceedings of the National Academy
  of Sciences, vol.~119, 2022, p.~e2107151119.

\bibitem{DR14}
{\sc M.~Donatelli and L.~Reichel}, {\em Square smoothing regularization
  matrices with accurate boundary conditions}, Journal of Computational and
  Applied Mathematics, 272 (2014), pp.~334--349.

\bibitem{eliasof2023drip}
{\sc M.~Eliasof, E.~Haber, and E.~Treister}, {\em {DRIP: deep regularizers for
  inverse problems}}, Inverse Problems, 40 (2023), p.~015006.

\bibitem{engl1996regularization}
{\sc H.~W. Engl, M.~Hanke, and A.~Neubauer}, {\em Regularization of inverse
  problems}, Springer Science \& Business Media, 1996.

\bibitem{coule-dataset}
{\sc D.~Evangelista, E.~Morotti, and E.~Loli~Piccolomini}, {\em {COULE}
  dataset}.
\newblock
  \url{https://www.kaggle.com/datasets/loiboresearchgroup/coule-dataset}, 2023.
\newblock Accessed on 01/12/2023.

\bibitem{evangelista2023rising}
{\sc D.~Evangelista, E.~Morotti, and E.~L. Piccolomini}, {\em {RISING: A new
  framework for model-based few-view CT image reconstruction with deep
  learning}}, Computerized Medical Imaging and Graphics, 103 (2023), p.~102156.

\bibitem{evangelista2023ambiguity}
{\sc D.~Evangelista, E.~Morotti, E.~L. Piccolomini, and J.~Nagy}, {\em
  Ambiguity in solving imaging inverse problems with deep-learning-based
  operators}, Journal of Imaging, 9 (2023), p.~133.

\bibitem{gazzola2019ir}
{\sc S.~Gazzola, P.~C. Hansen, and J.~G. Nagy}, {\em {IR Tools: a MATLAB
  package of iterative regularization methods and large-scale test problems}},
  Numerical Algorithms, 81 (2019), pp.~773--811.

\bibitem{gilboa2007nonlocal}
{\sc G.~Gilboa and S.~Osher}, {\em Nonlocal linear image regularization and
  supervised segmentation}, Multiscale Modeling \& Simulation, 6 (2007),
  pp.~595--630.

\bibitem{gilboa2009nonlocal}
{\sc G.~Gilboa and S.~Osher}, {\em Nonlocal operators with applications to
  image processing}, Multiscale Modeling \& Simulation, 7 (2009),
  pp.~1005--1028.

\bibitem{gottschling2020troublesome}
{\sc N.~M. Gottschling, V.~Antun, A.~C. Hansen, and B.~Adcock}, {\em {The
  troublesome kernel--On hallucinations, no free lunches and the
  accuracy-stability trade-off in inverse problems}}, 2020,
  \url{https://arxiv.org/abs/arXiv:2001.01258}.

\bibitem{han2016deep}
{\sc Y.~S. Han, J.~Yoo, and J.~C. Ye}, {\em Deep residual learning for
  compressed sensing {CT} reconstruction via persistent homology analysis},
  arXiv preprint arXiv:1611.06391,  (2016).

\bibitem{hansen1998rank}
{\sc P.~C. Hansen}, {\em {Rank-deficient and discrete ill-posed problems:
  numerical aspects of linear inversion}}, SIAM, 1998.

\bibitem{hansen2006deblurring}
{\sc P.~C. Hansen, J.~G. Nagy, and D.~P. O'leary}, {\em {Deblurring images:
  matrices, spectra, and filtering}}, SIAM, 2006.

\bibitem{hochstenbach2011fractional}
{\sc M.~E. Hochstenbach and L.~Reichel}, {\em {Fractional Tikhonov
  regularization for linear discrete ill-posed problems}}, BIT Numerical
  Mathematics, 51 (2011), pp.~197--215.

\bibitem{huang2013nonstationary}
{\sc J.~Huang, M.~Donatelli, and R.~H. Chan}, {\em Nonstationary iterated
  thresholding algorithms for image deblurring}, Inverse Probl. Imaging, 7
  (2013), pp.~717--736.

\bibitem{jin2017deep}
{\sc K.~H. Jin, M.~T. McCann, E.~Froustey, and M.~Unser}, {\em Deep
  convolutional neural network for inverse problems in imaging}, IEEE
  transactions on image processing, 26 (2017), pp.~4509--4522.

\bibitem{kak2001principles}
{\sc A.~C. Kak and M.~Slaney}, {\em Principles of computerized tomographic
  imaging}, SIAM, 2001.

\bibitem{keller2021graphs}
{\sc M.~Keller, D.~Lenz, and R.~K. Wojciechowski}, {\em {Graphs and Discrete
  Dirichlet Spaces}}, Grundlehren der mathematischen Wissenschaften, Springer,
  Cham, 2021.

\bibitem{klann2008regularization}
{\sc E.~Klann and R.~Ramlau}, {\em Regularization by fractional filter methods
  and data smoothing}, Inverse Problems, 24 (2008), p.~025018.

\bibitem{lanza2015generalized}
{\sc A.~Lanza, S.~Morigi, L.~Reichel, and F.~Sgallari}, {\em A generalized
  krylov subspace method for $\ell_p-\ell_q$ minimization}, SIAM Journal on
  Scientific Computing, 37 (2015), pp.~S30--S50.

\bibitem{li2020nett}
{\sc H.~Li, J.~Schwab, S.~Antholzer, and M.~Haltmeier}, {\em {NETT: Solving
  inverse problems with deep neural networks}}, Inverse Problems, 36 (2020),
  p.~065005.

\bibitem{lou2010image}
{\sc Y.~Lou, X.~Zhang, S.~Osher, and A.~Bertozzi}, {\em Image recovery via
  nonlocal operators}, Journal of Scientific Computing, 42 (2010),
  pp.~185--197.

\bibitem{miyato2018spectral}
{\sc T.~Miyato, T.~Kataoka, M.~Koyama, and Y.~Yoshida}, {\em {Spectral
  Normalization for Generative Adversarial Networks}}, in International
  Conference on Learning Representations, 2018.

\bibitem{moen2021low}
{\sc T.~R. Moen, B.~Chen, D.~R. Holmes~III, X.~Duan, Z.~Yu, L.~Yu, S.~Leng,
  J.~G. Fletcher, and C.~H. McCollough}, {\em {Low-dose CT image and projection
  dataset}}, Medical physics, 48 (2021), pp.~902--911.

\bibitem{morotti2021green}
{\sc E.~Morotti, D.~Evangelista, and E.~Loli~Piccolomini}, {\em A green
  prospective for learned post-processing in sparse-view tomographic
  reconstruction}, Journal of Imaging, 7 (2021), p.~139.

\bibitem{morotti2023increasing}
{\sc E.~Morotti, D.~Evangelista, and E.~L. Piccolomini}, {\em Increasing noise
  robustness of deep learning-based image processing with model-based
  approaches}, Numerical Computations: Theory and Algorithms NUMTA 2023,
  (2023), p.~155.

\bibitem{morozov2012methods}
{\sc V.~A. Morozov}, {\em Methods for solving incorrectly posed problems},
  Springer New York, NY, 1984.

\bibitem{mukherjee2020learned}
{\sc S.~Mukherjee, S.~Dittmer, Z.~Shumaylov, S.~Lunz, O.~{\"O}ktem, and C.-B.
  Sch{\"o}nlieb}, {\em {Learned Convex Regularizers for Inverse Problems}},
  2020, \url{https://arxiv.org/abs/arXiv:2008.02839}.

\bibitem{Peyre2008nonlocal}
{\sc G.~Peyr{\'{e}}, S.~Bougleux, and L.~Cohen}, {\em {Non-local regularization
  of inverse problems}}, in Computer Vision -- ECCV 2008, D.~Forsyth, P.~Torr,
  and A.~Zisserman, eds., Springer Berlin Heidelberg, 2008, pp.~57--68.

\bibitem{ronneberger2015u}
{\sc O.~Ronneberger, P.~Fischer, and T.~Brox}, {\em {U-net: Convolutional
  networks for biomedical image segmentation}}, in Medical Image Computing and
  Computer-Assisted Intervention--MICCAI 2015: 18th International Conference,
  Munich, Germany, October 5-9, 2015, Proceedings, Part III 18, Springer, 2015,
  pp.~234--241.

\bibitem{ROF92}
{\sc L.~I. Rudin, S.~Osher, and E.~Fatemi}, {\em Nonlinear total variation
  based noise removal algorithms}, Physica D: nonlinear phenomena, 60 (1992),
  pp.~259--268.

\bibitem{scherzer2009variational}
{\sc O.~Scherzer, M.~Grasmair, H.~Grossauer, M.~Haltmeier, and F.~Lenzen}, {\em
  {Variational methods in imaging}}, Springer, 2009.

\bibitem{schwab2019deep}
{\sc J.~Schwab, S.~Antholzer, and M.~Haltmeier}, {\em {Deep null space learning
  for inverse problems: convergence analysis and rates}}, Inverse Problems, 35
  (2019), p.~025008.

\bibitem{tan2024provably}
{\sc H.~Y. Tan, S.~Mukherjee, J.~Tang, and C.-B. Sch{\"o}nlieb}, {\em {Provably
  convergent plug-and-play quasi-Newton methods}}, SIAM Journal on Imaging
  Sciences, 17 (2024), pp.~785--819.

\bibitem{tjoa2020survey}
{\sc E.~Tjoa and C.~Guan}, {\em {A survey on explainable artificial
  intelligence (xai): Toward medical xai}}, IEEE transactions on neural
  networks and learning systems, 32 (2020), pp.~4793--4813.

\bibitem{wang2004image}
{\sc Z.~Wang, A.~C. Bovik, H.~R. Sheikh, and E.~P. Simoncelli}, {\em {Image
  quality assessment: from error visibility to structural similarity}}, IEEE
  transactions on image processing, 13 (2004), pp.~600--612.

\bibitem{zhang2010bregmanized}
{\sc X.~Zhang, M.~Burger, X.~Bresson, and S.~Osher}, {\em Bregmanized nonlocal
  regularization for deconvolution and sparse reconstruction}, SIAM Journal on
  Imaging Sciences, 3 (2010), pp.~253--276.

\end{thebibliography}
\end{document}